\documentclass[a4paper,oneside,10pt]{article}%
\usepackage[english]{babel}
\usepackage[T1]{fontenc}
\usepackage[a4paper,top=3cm,bottom=3cm,left=3cm,right=3cm,marginparwidth=1.75cm]%
{geometry}
\usepackage[colorlinks=true, allcolors=blue, hypertexnames=false]{hyperref}
\usepackage{empheq}
\usepackage{graphicx}
\usepackage[colorinlistoftodos]{todonotes}
\usepackage{amsmath}
\usepackage{mathtools}
\usepackage{upgreek}
\usepackage{amssymb}
\usepackage{amsfonts}
\usepackage{amsthm}
\usepackage{hyperref}
\usepackage{enumitem}
\usepackage{mathrsfs}
\usepackage{fontenc}
\usepackage{verbatim}
\usepackage{framed}
\usepackage{algorithm}
\usepackage{bbm}
\usepackage{algpseudocode}
\usepackage{appendix}
\usepackage{color}
\usepackage{multirow}%
\providecommand{\U}[1]{\protect\rule{.1in}{.1in}}

\newtheorem{theorem}{Theorem}[section]

\newtheorem{corollary}[theorem]{Corollary}

\newtheorem{assumption}[theorem]{Assumption}
\newtheorem{example}[theorem]{Example}

\newtheorem{lemma}[theorem]{Lemma}

\newtheorem{proposition}[theorem]{Proposition}
\newtheorem{remark}[theorem]{Remark}

\numberwithin{equation}{section}

\makeatletter
\newenvironment{breakablealgorithm}
  {
   \begin{center}
     \refstepcounter{algorithm}
     \hrule height.8pt depth0pt \kern2pt
     \renewcommand{\caption}[2][\relax]{
       {\raggedright\textbf{\ALG@name~\thealgorithm} ##2\par}%
       \ifx\relax##1\relax 
         \addcontentsline{loa}{algorithm}{\protect\numberline{\thealgorithm}##2}%
       \else 
         \addcontentsline{loa}{algorithm}{\protect\numberline{\thealgorithm}##1}%
       \fi
       \kern2pt\hrule\kern2pt
     }
  }{
     \kern2pt\hrule\relax
   \end{center}
  }
\makeatother

\begin{document}

\title{A Modified Method of Successive Approximations for Stochastic Recursive Optimal
Control Problems}
\author{Shaolin Ji\thanks{Zhongtai Securities Institute for Financial Studies,
Shandong University, Jinan, Shandong 250100, PR China. Email: jsl@sdu.edu.cn
(Corresponding author). Research supported by the National Natural Science
Foundation of China (No. 11971263; 11871458).}
\and Rundong Xu\thanks{Zhongtai Securities Institute for Financial Studies,
Shandong University, Jinan, Shandong 250100, PR China. Email:
rundong.xu@mail.sdu.edu.cn.}}
\maketitle

\textbf{Abstract}. Based on the stochastic maximum principle
for the partially coupled forward-backward stochastic
control system (FBSCS for short), a modified method of successive approximations
(MSA for short) is established for stochastic recursive optimal control problems.
The second-order adjoint processes are introduced in the
augmented Hamiltonian minimization step since the control domain is not necessarily convex. Thanks to the theory of bounded mean oscillation martingales (BMO martingales for short), we give a delicate proof of the error estimate and then prove the convergence of the modified MSA algorithm. In a special case, we obtain a logarithmic convergence rate.
When the control domain is convex and compact, a sufficient condition which makes the control returned from the MSA algorithm be a near-optimal control is given for a class of linear FBSCSs.

{\textbf{Key words}. BMO martingales; Forward-backward stochastic differential equations;
Method of successive approximations; Stochastic maximum principle;
Stochastic recursive optimal control }

\textbf{MSC subject classifications.} 93E20, 60H10, 60H30, 49M05

\addcontentsline{toc}{section}{\hspace*{1.8em}Abstract}

\section{Introduction}

Finding numerical solutions to
optimal control problems by scientific computing methods has attracted much attention in recent years.
As one of those methods, the method of successive approximations (MSA for short) is an
efficient tool to tackle optimal control problems. Compared
with the algorithms based on the dynamic programming approach (for example,
the Bellman-Howard policy iteration algorithm in \cite{BDL-Howard-2020}), the
MSA is an iterative method equipped with alternating propagation and
optimization steps based on the maximum principle.
New application of the modified MSA to a deep learning problem has been investigated
recently in \cite{WeinanE-2018},
which leads to an alternative approach to training the deep neural networks from the
deterministic optimal control viewpoint.

The MSA based on Pontryagin's maximum principle \cite{Pontryagin} for seeking
numerical solutions to deterministic control systems was first proposed by
Krylov et al. \cite{IA-Krylov1}. This method includes successive integrations
of the state and adjoint equations, and updates the control variables by
minimizing the Hamiltonian. After that, many improved modifications of the MSA
have been developed by researchers for a variety of deterministic control systems (\cite{AAL2, IA-Krylov2, WeinanE-2018, AAL1}).

A recent breakthrough in investigating the modified MSA for classical stochastic control systems
can be found in \cite{BDL-MSA-2020}, where the convergence result is based on the local stochastic maximum
principle (SMP for short, see Theorem 4.12 in \cite{Carmona}). This provides a policy-updating algorithm
to find the local optimal control candidates to the classical stochastic control systems, which improved their former result
in \cite{BDL-Howard-2020} that is only capable of handling such kind of controlled dynamics with no control variables in the diffusion part.
Nevertheless, in order to obtain the convergence of the modified MSA, the authors assumed "$D_{x}^{2}%
\sigma(\cdot)\equiv0$" to eliminate the impact of the unboundedness of $q^{u}$ (see (\ref{1st-adj-eq})) when they deduce the error estimates.
Furthermore, since their modified MSA is based on the local SMP, it may fail to deal with the case when the control
domain is non-convex. Thus, there are two natural questions that whether the above strong assumption can be weakened
and how to modify the MSA to be applicable to the control problems with general control domains.

To go a further step, it has become increasingly clear that the modified MSA calls for an extension from the classical stochastic
control system to a more general one with a non-convex control domain and weaker assumptions imposed on coefficients.
Therefore, the main goal of this paper is to establish the modified MSA for
the stochastic recursive optimal control problem which the state equation is described by a partially coupled forward-backward stochastic
differential equation (FBSDE for short, see \cite{HuYing-Peng95},
\cite{Ma-Yong-Protter}, \cite{Zhang17} and the references therein), and deduce the convergence of it.
This kind of wider optimal control problem is closely related to the stochastic differential
utility which plays an important role in the study of economic and financial fields such as the preference difference,
the asset pricing, and the continuous-time general equilibrium in security markets (see
\cite{Epstein92, Karoui97, Quenez03} and the references therein).

More than that, we study the general case when the control domain may be
non-convex. For this purpose, the construction of the modified MSA needs to base on
the general SMP for the forward-backward stochastic control system (FBCSC for short).
As for the general SMP, Peng \cite{Peng90} first
established the general SMP for classical stochastic control systems. Then,
numerous progress has been made for various stochastic control systems
(\cite{Fuhrman-Hu13, Peng93, Tang98, Tang-Li94, YongZhou}). Recently, Hu \cite{Hu17}
introduced two adjoint equations to obtain the SMP for FBSCSs governed
by partially coupled FBSDEs and solved the open problem proposed by Peng
\cite{Peng99}. Inspired by Hu's work, Hu et al. \cite{HuJiXue18}\ lately proposed a new
method to obtain the first and second-order variational equations which are
essentially fully coupled FBSDEs, and derived the SMP for fully coupled FBSCSs.

Our main contributions are as follows. Firstly, we established a modified MSA for stochastic recursive optimal
control problems subject to the partially coupled FBSCS
(\ref{state-eq}) with a general control domain and proved it converges to a local minimum of the original control problem,
which completely covers the results obtained in \cite{BDL-MSA-2020}. It is worth pointing out that the
challenge to obtain the desired error estimate is the unboundedness of the
solution $q^{u}$ to the adjoint equation (\ref{1st-adj-eq}). As mentioned
earlier, this technical difficulty was avoided if we impose the restrictive
assumption "$D_{x}^{2}\sigma(\cdot)\equiv0$". Fortunately, we found that the stochastic integral
$q^{u}\cdot W$ is a multi-dimensional BMO martingale, and substantially benefit from
the harmonic analysis on the space of BMO martingales developed for tackling certain backward stochastic
differential equations with unbounded coefficients by Delbaen and Tang \cite{Delbaen-Tang2010}.
Due to some useful inequalities, in particular the probabilistic version of
Fefferman's inequality, we obtained the error estimate which is critical to the convergence of our modified MSA.
This also indicates that we can remove the above unnecessary assumption
imposed on the diffusion coefficient by employing the BMO property of $q^{u}\cdot W$.

Secondly, in contrast to the classical stochastic control system, the emergence of $Z^{u}$ in the
backward state equation of (\ref{state-eq}) makes the error estimate more
difficult and complicated. By applying the Girsanov transformation,
the process (\ref{zeta}) disappears in the drift term under a new reference probability measure.
It should be emphasized that the BMO property of any martingale under this new reference probability
measure can be inherited from the corresponding one under the original probability measure.
Then we get the error estimate (\ref{est-Jv-Ju}) successfully. Furthermore, since the control domain need not be convex, the augmented
Hamiltonian contains the second-order adjoint process $P^{u}$ (see
(\ref{2nd-adj-eq})) whose boundedness is essential to obtain the error
estimate (\ref{est-Jv-Ju}). We proved that the boundedness of $P^{u}$ depends on the BMO property of $q^{u}\cdot W$.

Thirdly, as the number of the iterations $m$ increases, we obtain a $\frac{1}{m}$-order convergence rate of
the stochastic control system only driven by a forward stochastic differential equation and the cost functional is quadratic both in the state and control processes.
In addition, from the viewpoint of the near-optimality \cite{Zhou-XY}, we also prove the control returned from the MSA algorithm is near-optimal
for a class of linear forward-backward stochastic recursive problems independent of $z$, when the control domain is convex and compact.

The rest of the paper is organized as follows. In section 2, preliminaries and
the formulation of our problem are given. In section 3, we first show
properties of the solutions to the adjoint equations, and then state our main
results consisting of the error estimate and the convergence of our modified MSA algorithm.
The results about the convergence rate and the sufficient condition of the near-optimality are
also given as applications of the modified MSA algorithm.
In section 4, we provide numerical demonstrations to illustrate the general results.

\section{Preliminaries and Problem Formulation}

Fix a terminal value $T>0$ and three positive integers $n$, $d$
and $k$. Let $(\Omega,\mathcal{F},\mathbb{P})$ be a complete probability space
on which a standard $d$-dimensional Brownian motion $W=(W_{t}^{1},W_{t}%
^{2},...W_{t}^{d})_{t\in\lbrack0,T]}^{\intercal}$ is defined, and
$\mathbb{F}:\mathbb{=}\left\{  \mathcal{F}_{t}\right\}  _{t\in\lbrack0,T]}$ be
the $\mathbb{P}$-augmentation of the natural filtration generated by $W$.

Denote by $\mathbb{R}^{n}$ the $n$-dimensional real Euclidean space,
$\mathbb{R}^{n\times m}$ the set of $n\times m$ real matrices ($n,m\geq1$) and
$\mathbb{S}^{n\times n}$ the set of all $n\times n$ symmetric matrices. The
scalar product (resp. norm) of $A$, $B\in\mathbb{R}^{n\times m}$ is denoted by
$\left( A,B\right) =\mathrm{tr}\{AB^{\intercal}\}$ (resp.
$\left\vert A\right\vert =\sqrt{\mathrm{tr}\left\{  AA^{\intercal}\right\}  }%
$), where the superscript $^{\intercal}$ denotes the transpose of vectors or
matrices. Denote by $I_{n}$ the $n \times n$ identity matrix.

For any given $p,q\geq1$, we introduce the following Banach spaces.

\noindent$L_{\mathcal{F}_{T}}^{p}(\Omega;\mathbb{R}^{n})$: the space of $\mathcal{F}%
_{T}$-measurable $\mathbb{R}^{n}$-valued random variables $\xi$ such that
$\mathbb{E}\left[  \left\vert \xi\right\vert ^{p}\right]  <\infty$.

\noindent$L_{\mathcal{F}_{T}}^{\infty}(\Omega;\mathbb{R}^{n})$: the space of
$\mathcal{F}_{T}$-measurable $\mathbb{R}^{n}$-valued random variables $\xi$
such that $\underset{\omega\in\Omega}{\mathrm{ess~sup}}\left\vert \xi(\omega)
\right\vert <\infty$.

\noindent$L_{\mathcal{F}}^{\infty}([0,T];\mathbb{R}^{n})$: the space of $\mathbb{F}%
$-adapted $\mathbb{R}^{n}$-valued processes $\varphi$ defined on $[0,T]$ such that
\[
\left\Vert \varphi\right\Vert _{\infty}:=\underset{(t,\omega)\in
\lbrack0,T]\times\Omega}{\mathrm{ess~sup}}\left\vert \varphi_{t}\left(
\omega\right)  \right\vert <\infty.
\]

\noindent$\mathcal{S}_{\mathcal{F}}^{p}([0,T];\mathbb{R}^{n})$: the space of
$\mathbb{F}$-adapted $\mathbb{R}^{n}$-valued continuous processes $\varphi$ such that $\mathbb{E}\left[
\sup\limits_{t\in\lbrack0,T]}\left\vert \varphi_{t}\right\vert ^{p}\right]
<\infty$.

\noindent$\mathcal{H}_{\mathcal{F}}^{p}([0,T];\mathbb{R}^{n})$: the space of
$\mathbb{R}^{n}$-valued $\mathbb{F}$-martingales $M=\left(  M^{1},\ldots
,M^{n}\right)  ^{\intercal}$ with continuous pathes such that $M_{0}=0$ and $\left\Vert M\right\Vert
_{\mathcal{H}^{p}}:=\left\Vert \sqrt{\mathrm{tr}\left\{  \left\langle
M\right\rangle _{T}\right\}  }\right\Vert _{L^{p}}<\infty$, where%
\[
\left\langle M\right\rangle _{t}:=\left(  \left\langle M^{i},M^{j}%
\right\rangle _{t}\right)  _{1\leq i,j\leq n}\text{ \ for \ }t\in\lbrack0,T].
\]

\noindent$\mathcal{M}_{{}}^{p}(\mathbb{R}^{n\times d})$: the space of $\mathbb{R}%
^{n\times d}$-valued $\mathbb{F}$-progressively measurable processes $\varphi$
defined on $[0,T]$ such that%
\[
\left\Vert \varphi\right\Vert _{\mathcal{M}_{{}}^{p}}:=\left(  \mathbb{E}%
\left[  \left(  \int_{0}^{T}\left\vert \varphi_{t}\right\vert ^{2}dt\right)
^{\frac{p}{2}}\right]  \right)  ^{\frac{1}{p}}<\infty.
\]

\noindent$\mathrm{BMO}$: the space of processes $M\in\mathcal{H}_{\mathcal{F}}^{2}%
([0,T];\mathbb{R})$ such that%
\begin{equation}
\left\Vert M\right\Vert _{\mathrm{BMO}}:=\sup_{\tau}\left\Vert \left(  \mathbb{E}%
\left[  \left\langle M\right\rangle _{T}-\left\langle M\right\rangle _{\tau
}\mid\mathcal{F}_{\tau}\right]  \right)  ^{\frac{1}{2}}\right\Vert _{\infty
}<\infty, \label{def-BMO}%
\end{equation}
where the supremum is taken over all stopping times $\tau\in\mathcal{[}%
0,T\mathcal{]}$. Furthermore, one can replace $\tau$ with all deterministic
times $t\in\lbrack0,T]$ in definition (\ref{def-BMO}).

\noindent$\mathcal{K}(\mathbb{R}^{n\times d})$: the space of $\mathbb{R}^{n\times d}%
$-valued processes $\varphi\in\mathcal{M}_{{}}^{2}(\mathbb{R}^{n\times d})$
such that%
\begin{equation}
\left\Vert \varphi\right\Vert _{\mathcal{K}}:=\sup_{\tau}\left\Vert \left(
\mathbb{E}\left[  \int_{\tau}^{T}\left\vert \varphi_{s}\right\vert ^{2}%
ds\mid\mathcal{F}_{\tau}\right]  \right)  ^{\frac{1}{2}}\right\Vert _{\infty
}<\infty, \label{def-BMO-multi}%
\end{equation}
where the supremum is taken over all stopping times $\tau\in\mathcal{[}%
0,T\mathcal{]}$. Furthermore, one can replace $\tau$ with all deterministic
times $t\in\lbrack0,T]$ in definition (\ref{def-BMO-multi}).

We write $\mathrm{BMO}(\mathbb{Q})$ and $\mathcal{K}(\mathbb{R}^{n\times d}%
;\mathbb{Q})$ for any probability $\mathbb{Q}$ defined on $\left(
\Omega,\mathcal{F}\right)  $ whenever it is necessary to indicate the
underlying probability. For simplicity, if the underlying probability is
$\mathbb{P}$, we still use the notations $\mathrm{BMO}$ and $\mathcal{K}(\mathbb{R}%
^{n\times d})$.

\subsection{Some Notations and Results of BMO Martingales}

Here we list some notations and results of BMO martingales, which will be used
in this paper. We refer the readers to \cite{Delbaen-Tang2010}, \cite{HWY},
\cite{Kazamaki} the references therein for more details.

Denote by $\mathcal{E}\left(  M\right)  $ the Dol\'{e}ans-Dade exponential of
a continuous local martingale $M$, that is, $\mathcal{E}\left(  M_{t}\right)
=\exp\left\{  M_{t}-\frac{1}{2}\left\langle M\right\rangle _{t}\right\}  $ for
any $t\in\lbrack0,T].$ If $M\in \mathrm{BMO}$, then $\mathcal{E}(M)$ is a uniformly
integrable martingale (see Theorem 2.3 in \cite{Kazamaki}).

Let $H$ be an $\mathbb{R}^{d}$-valued $\mathbb{F}$-adapted process. Denote by
$H\cdot W$ the stochastic integral of $H$ with respect to the $d$-dimensional
Brownian motion $W$, that is, $\left(  H\cdot W\right)  _{t}:=\sum_{i=1}%
^{d}\int_{0}^{t}H_{s}^{i}dW_{s}^{i}$ for $t\in\lbrack0,T]$.

The following theorem plays an significant role in characterizing the duality
between $\mathcal{H}_{\mathcal{F}}^{1}([0,T];\mathbb{R})$ and $\mathrm{BMO}$.

\begin{theorem}
[\cite{HWY}, Theorem 10.18]
\label{thm-Fefferman-ineq}
Let $N\in\mathcal{H}_{\mathcal{F}}^{1}([0,T];\mathbb{R})$, $M\in \mathrm{BMO}$,
and $\varphi$ be an $\mathbb{F}$-progressive measurable process such that $\mathbb{E} \left[ \left( \int_{0}^{T} \left\vert \varphi_{t} \right\vert^{2} d\langle N \rangle_{t}
\right)^{\frac{1}{2}} \right]<\infty$. Then, for any stopping time $\tau$ in $[0,T]$,
\[
\mathbb{E}\left[  \int_{\tau}^{T}\left\vert \varphi_{s}\right\vert \left\vert
d\left\langle M,N\right\rangle _{s}\right\vert \mid\mathcal{F}_{\tau}\right]
\leq\sqrt{2}\mathbb{E}\left[  \left(  \int_{\tau}^{T}\left\vert \varphi
_{s}\right\vert ^{2}d\left\langle N\right\rangle _{s}\right)  ^{\frac{1}{2}%
}\mid\mathcal{F}_{\tau}\right]  \left\Vert M\right\Vert _{\mathrm{BMO}}.
\]
Particularly, when $\tau=0$ and $\varphi=1$, we have
\[
\mathbb{E}\left[  \int_{0}^{T} \left\vert
d\left\langle M,N\right\rangle _{s}\right\vert \right]
\leq\sqrt{2} \left\Vert M\right\Vert _{\mathrm{BMO}} \left\Vert N\right\Vert _{\mathcal{H}^{1}},
\]
which is well known as Fefferman's inequality.
\end{theorem}

For any $M \in \mathrm{BMO}$, the energy-type inequality for $\left\langle M
\right\rangle $ is a significant result commonly used in the BMO martingale theory
(see \cite{Kazamaki}). In essence, for any $\varphi\in\mathcal{K}%
(\mathbb{R}^{n\times d})$, $\mathbb{F}$-stopping time $\tau$ on $[0,T]$ and
$A\in\mathcal{F}_{\tau}$, we can apply Garsia's Lemma (\cite{HWY}, Lemma
10.35) to the continuous increasing process $\left(  1_{A}\int_{\tau}%
^{t}\left\vert \varphi_{s}\right\vert ^{2}ds\right)  _{t\in\lbrack\tau,T]}$ to
obtain the following energy-type inequality.

\begin{proposition}
[Energy inequality]\label{energy inequality}Let $\varphi\in\mathcal{K}%
(\mathbb{R}^{n\times d})$. Then, for any integer $m$ and $\mathbb{F}$-stopping
time $\tau$ on $[0,T]$, we have%
\[
\mathbb{E}\left[  \left(  \int_{\tau}^{T}\left\vert \varphi_{s}\right\vert
^{2}ds\right)  ^{m}\mid\mathcal{F}_{\tau}\right]  \leq m!\left\Vert
\varphi\right\Vert _{\mathcal{K}}^{2m}.
\]

\end{proposition}

Recall that the space BMO depends on the underlying probability measure. The
following lemma shows the equivalence of different BMO-norms under the
Girsanov transformation.

\begin{lemma}
[\cite{Hu-Tang2016}, Lemma A.4]\label{lem-equiv-BMO-norm} Let $K>0$ be a given
constant and $M$ be in BMO. Then, there are constants $c_{1}>0$ and
$c_{2}>0$ depending only on $K$ such that for any $N \in \mathrm{BMO}$ and $\left\Vert
N\right\Vert _{\mathrm{BMO}}\leq K$, we have
\[
c_{1}\left\Vert M\right\Vert _{\mathrm{BMO}}\leq\left\Vert \tilde{M}\right\Vert
_{\mathrm{BMO}(\mathbb{\tilde{P}})}\leq c_{2}\left\Vert M\right\Vert _{\mathrm{BMO}},
\]
where $\tilde{M}:=M-\left\langle M,N\right\rangle $ and $d\mathbb{\tilde{P}%
}:=\mathcal{E}\left(  N_{T}\right)  d\mathbb{P}$.
\end{lemma}

The following proposition is a more profound result by applying Fefferman's inequality.

\begin{proposition}
[\cite{Delbaen-Tang2010}, Lemma 1.4]
\label{prop-Hp-BMO}
Let $p\geq1$. Assume that $X\in\mathcal{S}_{\mathcal{F}}^{p}([0,T];\mathbb{R})$ and $M\in \mathrm{BMO}$.
Then, $X\cdot M\in\mathcal{H}_{\mathcal{F}}^{p}([0,T];\mathbb{R})$. Moreover,
we have the following estimate%
\begin{equation*}
\left\Vert X\cdot M\right\Vert _{\mathcal{H}^{p}}\leq\sqrt{2}\left\Vert
X\right\Vert _{\mathcal{S}^{p}}\left\Vert M\right\Vert _{\mathrm{BMO}}
\end{equation*}
for $p>1$ and%
\begin{equation*}
\left\Vert X\cdot M\right\Vert _{\mathcal{H}^{1}}\leq\left\Vert X\right\Vert
_{\mathcal{S}^{1}}\left\Vert M\right\Vert _{\mathrm{BMO}}.
\end{equation*}
\end{proposition}

\subsection{Problem Formulation}

Consider the following decoupled FBSCS:
\begin{equation}
\left\{
\begin{array}
[c]{rl}%
dX_{t}^{u}= & b(t,X_{t}^{u},u_{t})dt+\sigma(t,X_{t}^{u},u_{t})dW_{t},\\
dY_{t}^{u}= & -f(t,X_{t}^{u},Y_{t}^{u},Z_{t}^{u},u_{t})dt+\left(  Z_{t}%
^{u}\right)  ^{\intercal}dW_{t},\\
X_{0}^{u}= & x_{0},\ Y_{T}^{u}=\Phi(X_{T}^{u}),
\end{array}
\right.  \label{state-eq}%
\end{equation}
with the cost functional%
\begin{equation}
J(u(\cdot)):=Y_{0}^{u} \label{cost-func}%
\end{equation}
for a given $x_{0} \in\mathbb{R}^{n}$ and measurable functions $b:[0,T]\times
\mathbb{R}^{n}\times U\longmapsto\mathbb{R}^{n}$, $\sigma:[0,T]\times
\mathbb{R}^{n}\times U\longmapsto\mathbb{R}^{n\times d}$, $f:[0,T]\times
\mathbb{R}^{n}\times\mathbb{R}\times\mathbb{R}^{d}\times U\longmapsto
\mathbb{R}$ and $\Phi:\mathbb{R}^{n}\longmapsto\mathbb{R}$, where the control
domain $U$ is a nonempty subset of $\mathbb{R}^{k}$,
and the $\mathbb{F}$-adapted process $u(\cdot)$ is called an admissible control which takes
values in $U$ satisfying%
\begin{equation}
\sup_{t\in\lbrack0,T]}\mathbb{E}\left[  \left\vert u_{t}\right\vert
^{8}\right]  <\infty. \label{control-integrable}%
\end{equation}
Denote by $\mathcal{U}[0,T]$ the set of all admissible controls, and assume $\inf_{u(\cdot)\in\mathcal{U}[0,T]}J\left(  u(\cdot)\right) > -\infty$.
We want to find an optimal control $\bar{u}(\cdot) \in \mathcal{U}[0,T]$ reaching the minimum of (\ref{cost-func}) or, if the minimum cannot be reached,
an $\epsilon$-optimal control $u^{\epsilon}(\cdot)$ such that $J(u^{\epsilon}(\cdot)) \leq \inf_{u(\cdot)\in\mathcal{U}[0,T]}J\left(  u(\cdot)\right) + \epsilon$
for some given $\epsilon>0$.

For deterministic control systems, it has been shown that the basic MSA may
diverge when a bad initial value of control is chosen (see \cite{AAL2}) or the
feasibility errors blow up (see \cite{WeinanE-2018}). It can be observed that
Kerimkulov et al. \cite{BDL-MSA-2020} proposed directly a modified MSA for
classical stochastic control systems to ensure the convergence. To go a
further step, we are aimed at establishing a modified MSA for stochastic
recursive control optimal control problems and obtaining the related
convergence result.

Before giving the modified MSA algorithm for (\ref{state-eq}), we first
introduce the SMP for it.
Set
\[%
\begin{array}
[c]{l}%
b(\cdot)=\left(  b^{1}(\cdot),b^{2}(\cdot),\ldots,b_{{}}^{n}(\cdot)\right)
^{\intercal}\in\mathbb{R}^{n},\\
\sigma(\cdot)=\left(  \sigma^{1}(\cdot),\sigma^{2}(\cdot),\ldots,\sigma
^{d}(\cdot)\right)  \in\mathbb{R}^{n\times d},\\
\sigma^{i}(\cdot)=\left(  \sigma^{1i}(\cdot),\sigma^{2i}(\cdot),\ldots
,\sigma_{{}}^{ni}(\cdot)\right)  ^{\intercal}\in\mathbb{R}^{n},\text{  }i=1,2,\ldots,d
\end{array}
\]
and impose the following assumptions on the coefficients of (\ref{state-eq}):
\begin{assumption}
\label{assum-1} Let $L_{i}, i=1,2,3$ be given positive constants.

(i) $b$ and $\sigma$ are twice continuously differentiable with respect to
$x$. $b$, $\sigma$, $b_{x}$, $\sigma_{x}$, $b_{xx}$, $\sigma_{xx}$ are
continuous in $\left(  x,u\right)  $. $b_{x}$, $\sigma_{x}$, $b_{xx}$,
$\sigma_{xx}$ are bounded. $b$ and $\sigma$ are bounded by $L_{1}\left(
1+\left\vert x\right\vert +\left\vert u\right\vert \right)  $.

(ii) $\Phi$ is twice continuously differentiable with respect to $x$. $\Phi_{x}$, $\Phi_{xx}$ are
bounded, and $\Phi$ is bounded by $L_{2}\left(  1+\left\vert x\right\vert
\right)  $.

(iii) $f$ is twice continuously differentiable with respect to $(x,y,z)$. $f$
together with its gradient $Df$, Hessian matrix $D^{2}f$ with respect to $x$,
$y$, $z$ are continuous in $(x,y,z,u)$. $Df$, $D^{2}f$ are bounded, and $f$ is
bounded by $L_{3}\left(  1+\left\vert x\right\vert +\left\vert y\right\vert
+\left\vert z\right\vert +\left\vert u\right\vert \right)  $.
\end{assumption}

Let us fix a $u(\cdot)\in\mathcal{U}[0,T]$ arbitrarily. Under Assumption
\ref{assum-1}, thanks to \cite{Protter90} (Chapter V, Theorem 6) and Theorem
5.1 in \cite{Karoui97}, (\ref{state-eq}) admits a unique solution $(X_{{}}%
^{u},Y_{{}}^{u},Z_{{}}^{u})\in\mathcal{S}_{\mathcal{F}}^{2}([0,T];\mathbb{R}%
^{n})\times\mathcal{S}_{\mathcal{F}}^{2}([0,T];\mathbb{R})\times$
$\mathcal{M}_{{}}^{2}(\mathbb{R}^{d})$. We call $(X_{{}}^{u},Y_{{}%
}^{u},Z_{{}}^{u})$ the state trajectory corresponding to $u(\cdot)$.
Particularly, let $\bar{u}(\cdot)$ be an optimal control, $\left(  \bar
{X},\bar{Y},\bar{Z}\right)  $ be the corresponding state trajectory of
(\ref{state-eq}) and $\left(  \bar{p},\bar{q} \right)  $, $\left(  \bar
{P},\bar{Q} \right)  $ be the corresponding unique solution to the first-order
adjoint equation (\ref{1st-adj-eq}), the second-order adjoint equation
(\ref{2nd-adj-eq}) below respectively. The (stochastic) Hamiltonian
$H:[0,T]\times\Omega\times\mathbb{R}_{{}}^{n}\times\mathbb{R}\times\mathbb{R}^{d}\times\mathbb{R}_{{}}^{n}%
\times\mathbb{R}_{{}}^{n\times d}\times\mathbb{R}_{{}}^{n\times n}\times
U\longmapsto\mathbb{R}$ is defined as follows:
\begin{align*}
& H(t,x,y,z,p,q,P,u)\\
& =\dfrac{1}{2}\sum\limits_{i=1}^{d}\left(  \sigma^{i}(t,x,u)-\sigma
^{i}(t,\bar{X}_{t},\bar{u}_{t})\right)  ^{\intercal}P\left(  \sigma
^{i}(t,x,u)-\sigma^{i}(t,\bar{X}_{t},\bar{u}_{t})\right)  \\
& +p^{\intercal}b(t,x,u)+\sum\limits_{i=1}^{d}\left(  q^{i}\right)
^{\intercal}\sigma_{{}}^{i}(t,x,u)+f(t,x,y,z+\Delta(t,x,u),u),
\end{align*}
where $q^{i}$ is the $i$th column of $q$ for $i=1,2,\ldots,d$, and
\[
\Delta(t,x,u) = \left(  \left(  \sigma_{{}}^{1}(t,x,u)-\sigma_{{}}^{1}%
(t,\bar{X}_{t},\bar{u}_{t})\right)  ^{\intercal}p,\ldots,\left(  \sigma_{{}%
}^{d}(t,x,u)-\sigma_{{}}^{d}(t,\bar{X}_{t},\bar{u}_{t})\right)  ^{\intercal
}p\right)  ^{\intercal}.
\]
Then the following stochastic maximum principle (\cite{Hu17}, Theorem 3; \cite{HuJiXue18}, Theorem 3.17) holds.

\begin{theorem}
\label{thm-maximum principle} Let Assumption \ref{assum-1} hold. Then, for all
$u\in U$,%
\begin{equation}
\label{Hu-SMP}H(t,\bar{X}_{t},\bar{Y}_{t},\bar{Z}_{t},\bar{p}_{t},\bar{q}%
_{t},\bar{P}_{t},u)\geq H(t,\bar{X}_{t},\bar{Y}_{t},\bar{Z}_{t}%
,\bar{p}_{t},\bar{q}_{t},\bar{P}_{t},\bar{u}_{t}),\ dt \otimes d\mathbb{P}\text{-}a.e..
\end{equation}

\end{theorem}

Secondly, it follows from the pioneering works mentioned before that a key
step to control the divergent behavior rigorously of the modified MSA is to
obtain the error estimate by estimating the difference between two cost
functionals $J(u(\cdot))$ and $J(v(\cdot))$ corresponding to different
admissible controls $u(\cdot)$ and $v(\cdot)$. In order to do this, we need to introduce
the following notations and the augmented Hamiltonian.

Define the function $G:[0,T]\times\mathbb{R}_{{}}^{n}\times\mathbb{R}%
\times\mathbb{R}_{{}}^{d}\times\mathbb{R}_{{}}^{n}\times\mathbb{R}_{{}%
}^{n\times d}\times U\times U\longmapsto\mathbb{R}$ by%
\[
G(t,x,y,z,p,q,v,u)=p_{{}}^{\intercal}b(t,x,v)+\sum\limits_{i=1}^{d}\left(
q^{i}\right)  ^{\intercal}\sigma^{i}(t,x,v)+f(t,x,y,z+\tilde{\Delta
}(t,x,p,v,u),v),
\]
where $q=\left(  q^{1},\ldots,q^{d}\right)  $ and%
\[
\tilde{\Delta}(t,x,p,v,u):=\left(  \left(  \sigma_{{}}^{1}(t,x,v)-\sigma_{{}%
}^{1}(t,x,u)\right)  ^{\intercal}p,\ldots,\left(  \sigma_{{}}^{d}%
(t,x,v)-\sigma_{{}}^{d}(t,x,u)\right)  ^{\intercal}p\right)  ^{\intercal}.
\]
Let $u(\cdot), v(\cdot) \in\mathcal{U}[0,T]$. For $\psi=b$,
$\sigma$, and $w=x$, $y$, $z$, we simply set%
\begin{equation}%
\begin{array}
[c]{ll}%
\Theta_{t}^{u}=\left(  X_{t}^{u},Y_{t}^{u},Z_{t}^{u}\right)  , & \Theta
_{t}^{u,v}=\left(  X_{t}^{u},Y_{t}^{u},Z_{t}^{u}+\tilde{\Delta}(t,X_{t}%
^{u},p_{t}^{u},v_{t},u_{t})\right)  ,\\
\psi^{u}(t)=\psi(t,X_{t}^{u},u_{t}), & \psi_{x}^{u}(t)=\psi_{x}(t,X_{t}%
^{u},u_{t}),\\
\psi_{xx}^{u}(t)=\psi_{xx}(t,X_{t}^{u},u_{t}), & \\
f^{u}(t)=f(t,\Theta_{t}^{u},u_{t}), & f^{u,v}(t)=f(t,\Theta_{t}^{u,v}%
,v_{t}),\\
f_{w}^{u}(t)=f_{w}(t,\Theta_{t}^{u},u_{t}), & f_{w}^{u,v}(t)=f_{w}%
(t,\Theta_{t}^{u,v},v_{t}),\\
f_{ww}^{u}(t)=f_{ww}(t,\Theta_{t}^{u},u_{t}), & f_{ww}^{u,v}(t)=f_{ww}%
(t,\Theta_{t}^{u,v},v_{t})
\end{array}
\label{def-notation}%
\end{equation}
and $D^{2}f^{u}(t)=D^{2}f(t,\Theta_{t}^{u},u_{t})$ for all $t\in\lbrack0,T]$.
In particular, for $i=1,\ldots,d$, we denote $\sigma_{x}^{u,i}(t)=\sigma_{x}^{i}(t,X_{t}^{u},u_{t})$
and $\sigma_{xx}^{u,i}(t)=\sigma_{xx}^{i}(t,X_{t}^{u},u_{t})$.

In our context, for $t\in[0,T]$, the first-order (resp. second-order) adjoint equation in
\cite{Hu17} can be rewritten as (\ref{1st-adj-eq}) (resp. (\ref{2nd-adj-eq})) below.
\begin{equation}%
\begin{array}
[c]{rl}%
p_{t}^{u}= & \Phi_{x}(X_{T}^{u})+\int_{t}^{T}\left\{  G_{x}(s,\Theta_{s}%
^{u},p_{s}^{u},q_{s}^{u},u_{s},u_{s})+G_{y}(s,\Theta_{s}^{u},p_{s}^{u}%
,q_{s}^{u},u_{s},u_{s})p_{s}^{u}\right.  \\
& +\left.  \Upsilon(s,X_{s}^{u},p_{s}^{u},q_{s}^{u},u_{s})G_{z}(s,\Theta
_{s}^{u},p_{s}^{u},q_{s}^{u},u_{s},u_{s})\right\}  ds-\sum\limits_{i=1}%
^{d}\int_{t}^{T}\left(  q_{s}^{u}\right)  ^{i}dW_{s}^{i},
\end{array}
\label{1st-adj-eq}%
\end{equation}

\begin{equation}%
\begin{array}
[c]{cl}%
P_{t}^{u}= & \Phi_{xx}(X_{T}^{u})+\int_{t}^{T}\left\{  f_{y}^{u}(s)P_{s}%
^{u}+\left(  b_{x}^{u}(s)\right)  ^{\intercal}P_{s}^{u}+\left(  P_{s}%
^{u}\right)  ^{\intercal}b_{x}^{u}(s)\right.  \\
& +\sum\limits_{i=1}^{d}f_{z_{i}}^{u}(s)\left[  \left(  \sigma_{x}%
^{u,i}(s)\right)  ^{\intercal}P_{s}^{u}+\left(  P_{s}^{u}\right)  ^{\intercal
}\sigma_{x}^{u,i}(s)\right]  \\
& +\sum\limits_{i=1}^{d}\left(  \sigma_{x}^{u,i}(s)\right)  ^{\intercal}%
P_{s}^{u}\sigma_{x}^{u,i}(s)+\sum\limits_{i=1}^{d}f_{z_{i}}^{u}(s)\left(
Q_{s}^{u}\right)  ^{i}\\
& +\left.  \sum\limits_{i=1}^{d}\left[  \left(  \sigma_{x}^{u,i}(s)\right)
^{\intercal}\left(  Q_{s}^{u}\right)  ^{i}+\left(  \left(  Q_{s}^{u}\right)
^{i}\right)  ^{\intercal}\sigma_{x}^{u,i}(s)\right]  +\Psi_{s}^{u}\right\}
ds\\
& -\sum\limits_{i=1}^{d}\int_{t}^{T}\left(  Q_{s}^{u}\right)  ^{i}dW_{s}^{i},
\end{array}
\label{2nd-adj-eq}%
\end{equation}
where
\[
\Upsilon(t,X_{t}^{u},p_{t}^{u},q_{t}^{u},u_{t})=\left(  \left(  \sigma
_{x}^{1}(t,X_{t}^{u},u_{t})\right)  ^{\intercal}p_{t}^{u}+\left(  q_{t}%
^{u}\right)  ^{1},\ldots,\left(  \sigma_{x}^{d}(t,X_{t}^{u},u_{t})\right)
^{\intercal}p_{t}^{u}+\left(  q_{t}^{u}\right)  ^{d}\right),
\]

\begin{equation}%
\begin{array}
[c]{rl}%
\Psi_{t}^{u}= & \sum\limits_{j=1}^{n}\left(  b_{xx}^{u}(t)\right)
^{j}\left( p_{t}^{u}\right)  ^{j}+\sum\limits_{i=1}^{d}\sum\limits_{j=1}%
^{n}\left(  \sigma_{xx}^{u,i}(t)\right)  ^{j}\left(  f_{z_{i}}^{u}(t)\left(
p_{t}^{u}\right)  ^{j}+\left(  q_{t}^{u}\right)  ^{ji}\right) \\
& +\left(  I_{n},p_{t}^{u},\Upsilon(t,X_{t}^{u},p_{t}^{u},q_{t}^{u}%
,u_{t})\right)  D^{2}f^{u}(t)\left(  I_{n},p_{t}^{u},\Upsilon(t,X_{t}%
^{u},p_{t}^{u},q_{t}^{u},u_{t})\right)  ^{\intercal}.
\end{array}
\label{2nd-adj-eq-multi-data}%
\end{equation}
Here, for $i=1,\ldots,d$ and $j=1,\ldots,n$, $\left( p^{u}\right)  ^{j}$, $\left(  q^{u}\right)  ^{ji}$ are the $j$th components of $p^{u}$, $\left(  q^{u}\right)  ^{i}$
respectively;
$\left(  b_{xx}^{u}(t)\right)^{j}$, $\left(  \sigma_{xx}^{u,i}(t)\right)  ^{j}$ are the Hessian matrices of the $j$th components of
$b^{u}(t)$, $\sigma^{u,i}(t)$ respectively.

Define the (deterministic) Hamiltonian $\mathcal{H}:[0,T]\times\mathbb{R}_{{}}%
^{n}\times\mathbb{R}\times\mathbb{R}^{d}\times\mathbb{R}_{{}}^{n}\times\mathbb{R}_{{}}^{n\times d}\times
\mathbb{R}_{{}}^{n\times n}\times U\times U\longmapsto\mathbb{R}$ by%
\begin{equation}%
\begin{array}
[c]{ll}
& \mathcal{H}(t,x,y,z,p,q,P,v,u)\\
= & G(t,x,y,z,p,q,v,u)\\
& +\dfrac{1}{2}\sum\limits_{i=1}^{d}\left(  \sigma^{i}(t,x,v)-\sigma
^{i}(t,x,u)\right)  ^{\intercal}P\left(  \sigma^{i}(t,x,v)-\sigma
^{i}(t,x,u)\right)  .
\end{array}
\label{def-Hamiltonian}%
\end{equation}
Then, for all $u \in U$, the SMP in Theorem
\ref{thm-maximum principle} can be rewritten as follows:
\[
\mathcal{H}(t,\bar{X}_{t},\bar{Y}_{t},\bar{Z}_{t},\bar{p}_{t},\bar{q}_{t}%
,\bar{P}_{t},u,\bar{u}_{t})\geq\mathcal{H}(t,\bar{X}_{t},\bar{Y}%
_{t},\bar{Z}_{t},\bar{p}_{t},\bar{q}_{t},\bar{P}_{t},\bar{u}_{t},\bar{u}_{t}%
),\text{ } dt \otimes d\mathbb{P}\text{-} a.e..
\]
Now we introduce the augmented Hamiltonian $\mathcal{\tilde{H}}:[0,T]\times\mathbb{R}_{{}}^{n}
\times\mathbb{R}\times\mathbb{R}^{d}\times\mathbb{R}_{{}}^{n}\times\mathbb{R}_{{}}^{n\times
d}\times\mathbb{R}_{{}}^{n\times n}\times U\times U\longmapsto\mathbb{R}$ for
some $\rho\geq0$ by%
\begin{equation}%
\begin{array}
[c]{cl}
& \mathcal{\tilde{H}(}t,x,y,z,p,q,P,v,u\mathcal{)}\\
= & \mathcal{H}(t,x,y,z,p,q,P,v,u)+\dfrac{\rho}{2}\left\{  \sum\limits_{\psi
\in\{b,\sigma,f\}}\left\vert \psi(t,x,y,z,v)-\psi(t,x,y,z,u)\right\vert
^{2}\right.  \\
& +\left.  \sum\limits_{w\in\{x,y,z\}}\left\vert G_{w}(t,x,y,z,p,q,v,u)-G_{w}%
(t,x,y,z,p,q,u,u)\right\vert ^{2}\right\}  .
\end{array}
\label{def-general-Hamiltonian}%
\end{equation}
Note that when $\rho=0$ we get exactly the Hamiltonian
(\ref{def-Hamiltonian}). Moreover, the SMP also holds for
$\mathcal{\tilde{H}}$, which is a basis of constructing the iterations in the MSA algorithm.

\begin{lemma}
[\textbf{Extended SMP}]Let $\bar{u}(\cdot)$ be an optimal control, $\left(
\bar{X},\bar{Y},\bar{Z}\right)  $ be the corresponding state trajectory of
(\ref{state-eq}) and $\left(  \bar{p},\bar{q} \right)  $, $\left(  \bar
{P},\bar{Q} \right)  $ be the corresponding unique solutions to
(\ref{1st-adj-eq}), (\ref{2nd-adj-eq}) respectively. Then we have, $ dt \otimes d\mathbb{P}\text{-}a.e.$,
\begin{equation}
\mathcal{\tilde{H}}(t,\bar{X}_{t},\bar{Y}_{t},\bar{Z}_{t},\bar{p}_{t},\bar
{q}_{t},\bar{P}_{t},\bar{u}_{t},\bar{u}_{t})=\min_{u\in U}\mathcal{\tilde{H}%
}(t,\bar{X}_{t},\bar{Y}_{t},\bar{Z}_{t},\bar{p}_{t},\bar{q}_{t},\bar{P}%
_{t},u,\bar{u}_{t}).\label{aug-SMP}%
\end{equation}
\end{lemma}

The proof of the extended SMP is a direct application of Theorem
\ref{thm-maximum principle} and (\ref{def-general-Hamiltonian}) so we omit it.
It should be emphasized that not all the multiples such as $\left(  \bar{X},\bar{Y},\bar{Z},\bar{p},\bar{q},\bar{P}, \bar{u}(\cdot)\right)$ satisfying (\ref{aug-SMP}) are globally optimal for (\ref{state-eq})-(\ref{cost-func}) since (\ref{Hu-SMP}) is only the
necessary condition of the optimality and (\ref{aug-SMP}) is weaker than (\ref{Hu-SMP}).
Nonetheless, we will lately obtain an ``approximation'' form of (\ref{aug-SMP}) in Theorem \ref{thm-convergence}, which
is helpful for us to derive the sufficient condition of the near-optimality for a class of linear forward-backward stochastic recursive control problems
and it will be proved rigorously in Theorem \ref{thm-suffi-cond}.

As the end of this section, we introduce the modified MSA in the following
algorithm.
\begin{breakablealgorithm}
\caption{The Modified Method of Successive Approximations for Stochastic Recursive Optimal Control Problems}
\label{algorithm-MSA}
\begin{algorithmic}[1]
\State \textbf{Variable Initialisation:} Put $m=1$, and take any $u^{0}(\cdot) \in \mathcal{U}[0,T]$ to be an initial approximation.
\State Solve the FBSDE (\ref{state-eq}) corresponding to $u^{0}(\cdot)$ and then obtain the state trajectory $\left ( X^{1},Y^{1},Z^{1} \right )$.
\State Calculate $J(u^{0}(\cdot))$.
\State Solve the BSDE (\ref{1st-adj-eq}) and (\ref{2nd-adj-eq}) with the control $u^{0}(\cdot)$ and state trajectory $\left ( X^{1},Y^{1},Z^{1} \right )$,
and then obtain the first and second-order adjoint processes $\left ( p^{1},q^{1} \right )$ and $\left ( P^{1},Q^{1} \right )$.
\State Set
\[
u_{t}^{1} \in \arg\min_{u\in U}\mathcal{\tilde{H}}(t,X_{t}^{1},Y_{t}^{1},Z_{t}^{1},p_{t}^{1},q_{t}^{1},P_{t}^{1},u,u_{t}^{0}), \quad t\in [0,T].
\]
\State Repeat steps 2-3 by replacing $u^{0}(\cdot)$ with $u^{1}(\cdot)$, and then obtain $\left ( X^{2},Y^{2},Z^{2} \right )$, $J(u^{1}(\cdot))$.
\While {$J(u^{m-1}(\cdot)) - J(u^{m}(\cdot))$ is larger than some given permissible error $\epsilon>0$}
\State Increase the value of $m$ by $1$.
\State Repeat the step 4 by replacing $u^{0}(\cdot)$, $\left ( X^{1},Y^{1},Z^{1} \right )$ with $u^{m-1}(\cdot)$, $\left ( X^{m},Y^{m},Z^{m} \right )$ respectively,
and then obtain $\left ( p^{m},q^{m} \right )$ and $\left ( P^{m},Q^{m} \right )$.
\State Update the control
\begin{equation}
u_{t}^{m} \in \arg\min_{u\in U}\mathcal{\tilde{H}}(t,X_{t}^{m},Y_{t}^{m},Z_{t}^{m},p_{t}^{m},q_{t}^{m},P_{t}^{m},u,u_{t}^{m-1}) \quad, t\in\lbrack0,T].
\label{def-u-n-step}%
\end{equation}
\State Repeat steps 2-3 by replacing $u^{0}(\cdot)$ with $u^{m}(\cdot)$, and then obtain $\left ( X^{m+1},Y^{m+1},Z^{m+1} \right )$, $J(u^{m}(\cdot))$.
\EndWhile
\State \Return $u_{}^{m-1}$.
\end{algorithmic}
\end{breakablealgorithm}

\section{Main Results}

In this section, the universal constant $C$ may depend only on $n$, $d$, $T$,
$\left\Vert b_{x}\right\Vert _{\infty}$, $\left\Vert \sigma_{x}\right\Vert
_{\infty}$, $\left\Vert \Phi_{x}\right\Vert _{\infty}$, $\left\Vert
Df\right\Vert _{\infty}$, $\left\Vert D^{2}f\right\Vert _{\infty}$ and\ will
change from line to line in our proof.

\subsection{Properties of Solutions to Adjoint Equations}

Before stating the main results of the paper, we first give some properties of
solutions to the adjoint equations (\ref{1st-adj-eq}) and (\ref{2nd-adj-eq}),
which are necessary to prove the convergence of the Algorithm
\ref{algorithm-MSA}.

Under Assumption \ref{assum-1}, the following lemma shows that the solution $\left( p^{u}, q^{u} \right)$ to the first-order
adjoint equation (\ref{1st-adj-eq}) is uniformly bounded in $L_{\mathcal{F}}^{\infty}([0,T];\mathbb{R}^{n}) \times \mathcal{K}(\mathbb{R}^{n\times d})$
across all admissible controls.

\begin{lemma}
\label{adj-1st-lem} Let Assumption \ref{assum-1} hold. Then, for any
$u(\cdot)\in\mathcal{U}[0,T]$, (\ref{1st-adj-eq}) admits a unique
solution $\left(  p^{u},q^{u}\right)  \in\mathcal{S}_{\mathcal{F}}%
^{2}([0,T];\mathbb{R}_{{}}^{n})\times\mathcal{M}_{{}}^{2}(\mathbb{R}^{n\times
d})$. Moreover, we have
\[
\sup_{u(\cdot)\in\mathcal{U}[0,T]}\left(  \left\Vert p_{{}}^{u}\right\Vert
_{\infty}+\left\Vert q_{{}}^{u}\right\Vert _{\mathcal{K}}\right)  \leq C,
\]
where $C$ is independent of $u(\cdot)$.
\end{lemma}

\begin{proof}[\textbf{Proof}]
At first, for any given $u(\cdot)\in\mathcal{U}[0,T]$, (\ref{1st-adj-eq}) can be rewritten in the following form:%
\begin{equation}%
\begin{array}
[c]{rl}%
p_{t}^{u}= & \Phi_{x}(X_{T}^{u})+\int_{t}^{T}\left[  \left(  A_{1}%
^{u}(s)\right)  ^{\intercal}p_{s}^{u}+\sum\limits_{i=1}^{d}\left(
B_{1}^{u,i}(s)\right)  ^{\intercal}\left(  q_{s}^{u}\right)
^{i}+f_{x}^{u}(s)\right]  ds\\
& -\sum\limits_{i=1}^{d}\int_{t}^{T}\left(  q_{s}^{u}\right)  ^{i}dW_{s}%
^{i},\text{ \ \ \ }t\in\lbrack0,T],
\end{array}
\label{1st-adj-eq-linear-form}%
\end{equation}
where%
\[%
\begin{array}
[c]{l}%
A_{1}^{u}(t)=\sum\limits_{i=1}^{d}f_{z_{i}}^{u}(t)  \sigma_{x}%
^{u,i}(t)+f_{y}^{u}(t)I_{n}+b_{x}^{u}(t);\\
B_{1}^{u,i}(t)=f_{z_{i}}^{u}(t)I_{n}+ \sigma_{x} %
^{u,i}(t),\text{ \ for \ }i=1,2,\ldots,d;\\
q^{u}=\left(  \left(  q_{{}}^{u}\right)  ^{1},\left(  q_{{}}%
^{u}\right)  ^{2},\ldots,\left(  q_{{}}^{u}\right)  ^{d}\right).
\end{array}
\]
Since $b_{x}$, $\sigma_{x}$, $\Phi_{x}$, $f_{x}$, $f_{y}$ and $f_{z}$ are
bounded, it can be easily verified that both $A^{u}(\cdot)$ and $B^{u}(\cdot)$
are uniformly bounded. Moreover, by Theorem 5.1 in \cite{Karoui97}, (\ref{1st-adj-eq}) admits a unique solution
$\left( p^{u},q^{u}\right)  \in\mathcal{S}_{\mathcal{F}}^{2}([0,T];\mathbb{R}_{{}}%
^{n})\times\mathcal{M}_{{}}^{2}(\mathbb{R}^{n\times d})$. Furthermore, $p^{u}$ can be expressed explicitly by%
\begin{equation}
p_{t}^{u}=\mathbb{E}\left[  \left(  \Lambda_{t}^{u}\right)  ^{\intercal
}\left(  \Gamma_{T}^{u}\right)  ^{\intercal}\Phi_{x}(X_{T}^{u})+\int_{t}%
^{T}\left(  \Lambda_{t}^{u}\right)  ^{\intercal}\left(  \Gamma_{s}^{u}\right)
^{\intercal}f_{x}^{u}(s)ds\mid\mathcal{F}_{t}\right], \quad t\in
\lbrack0,T]\label{p-explicit}%
\end{equation}
(see Chapter 7, Section 2, Theorem 2.2 in \cite{YongZhou}),
where $\Gamma^{u}$ and $\Lambda^{u}$ satisfy the following matrix-valued SDEs respectively:
\[
\Gamma_{t}^{u}=I_{n}+\int_{0}^{t}A_{1}^{u}(s)\Gamma_{s}^{u}ds+\sum_{i=1}%
^{d}\int_{0}^{t}B_{1}^{u,i}(s)\Gamma_{s}^{u}dW_{s}%
^{i}, \quad \text{ \ } t\in\lbrack0,T]
\]
and%
\[%
\begin{array}
[c]{rl}%
\Lambda_{t}^{u}= & I_{n}+%
{\displaystyle\int_{0}^{t}}
\Lambda_{s}^{u}\left[  -A_{1}^{u}(s)+\sum\limits_{i=1}^{d}\left(
B_{1}^{u,i}(s)\right)  ^{2}\right]  ds\\
& -\sum\limits_{i=1}^{d}%
{\displaystyle\int_{0}^{t}}
\Lambda_{s}^{u}  B_{1}^{u,i}(s)dW_{s}^{i}, \quad \text{ \ }%
t\in\lbrack0,T].
\end{array}
\]
By using It\^{o}'s formula, one can verify that $\Lambda_{t}^{u}=\left(
\Gamma_{t}^{u}\right)  ^{-1}$
$\mathbb{P}$-almost surely for all $t\in\lbrack0,T]$. For each fixed
$t\in\lbrack0,T]$, set $\left(  \Gamma_{s}^{t}\right)  ^{u}=\Gamma_{s}%
^{u}\Lambda_{t}^{u}$ for $s\in\lbrack t,T]$. Then it is easy to check that
$\Gamma^{t}$ satisfies the following SDE:%
\begin{equation}
\left(  \Gamma_{s}^{t}\right)  ^{u}=I_{n}+%
{\displaystyle\int_{t}^{s}}
A_{1}^{u}(s)\left(  \Gamma_{r}^{t}\right)  ^{u}dr+\sum_{i=1}^{d}%
{\displaystyle\int_{t}^{s}}
B_{1}^{u,i}(s)\left(  \Gamma_{r}^{t}\right)  ^{u}dW_{r}%
^{i}, \text{    } s \in [t,T].
\label{multi-dynamic-SDE}%
\end{equation}
By using a standard SDE estimate, we obtain $\mathbb{E}\left[  \sup
_{s\in\lbrack t,T]}\left\vert \left(  \Gamma_{s}^{t}\right)  ^{u}\right\vert
^{\beta}\mid\mathcal{F}_{t}\right]  \leq C$ for all $t\in\lbrack0,T]$ and any
$\beta>1$. Then, it follows immediately that%
\[%
\begin{array}
[c]{cl}%
\left\vert p_{t}^{u}\right\vert  & \leq\mathbb{E}\left[  \sup\limits_{s\in
\lbrack t,T]}\left\vert \left(  \Gamma_{s}^{t}\right)  ^{u}\right\vert \left(
\left\vert \Phi_{x}(X_{T}^{u})\right\vert +%
{\displaystyle\int_{t}^{T}}
\left\vert f_{x}^{u}(s)\right\vert ds\right)  \mid\mathcal{F}_{t}\right]  \\
& \leq\left(  \left\Vert \Phi_{x}\right\Vert _{\infty}+\left\Vert
f_{x}\right\Vert _{\infty}T\right)  \left(  \mathbb{E}\left[  \sup
\limits_{s\in\lbrack t,T]}\left\vert \left(  \Gamma_{s}^{t}\right)
^{u}\right\vert ^{2}\mid\mathcal{F}_{t}\right]  \right)  ^{\frac{1}{2}}\\
& \leq C,
\end{array}
\]
where $C$ is independent of $u(\cdot)$.
Hence, we deduce that $\sup_{u(\cdot)\in\mathcal{U}[0,T]}\left\Vert p_{{}}^{u}\right\Vert _{\infty}<\infty$.
Moreover, applying It\^{o}'s formula to $\left \vert p_{t}^{u} \right \vert^{2}$ on $[t,T]$, since $p^{u}$ is bounded, one can obtain that%
\begin{equation}%
\begin{array}
[c]{rl}
&
{\displaystyle\int_{t}^{T}}
\left\vert q_{s}^{u}\right\vert ^{2}ds\\
\leq & \left\vert \Phi_{x}(X_{T}^{u})\right\vert ^{2}+2%
{\displaystyle\int_{t}^{T}}
\left\vert \left(  p_{s}^{u},\left(  A_{1}^{u}(s)\right)  ^{\intercal}%
p_{s}^{u}\right)  \right\vert ds+2%
{\displaystyle\int_{t}^{T}}
\left\vert \left(  p_{s}^{u},f_{x}^{u}(s)\right)  \right\vert ds\\
& +2\sum\limits_{i=1}^{d}%
{\displaystyle\int_{t}^{T}}
\left\vert \left(  p_{s}^{u},\left(  B_{1}^{u,i}(s)\right)  ^{\intercal
}\left(  q_{s}^{u}\right)  ^{i}\right)  \right\vert ds-2\sum\limits_{i=1}^{d}%
{\displaystyle\int_{t}^{T}}
\left(  p_{s}^{u},\left(  q_{s}^{u}\right)  ^{i}\right)  dW_{s}^{i},\\
\leq & \left\Vert \Phi_{x}\right\Vert _{\infty}^{2}+2T\left(  \left\Vert
A_{1}^{u}(\cdot)\right\Vert _{\infty}\left\Vert p_{{}}^{u}\right\Vert
_{\infty}^{2}+\left\Vert f_{x}\right\Vert _{\infty}\left\Vert p_{{}}%
^{u}\right\Vert _{\infty}\right)  \\
& +2\left\Vert p_{{}}^{u}\right\Vert _{\infty}\sum\limits_{i=1}^{d}\left\Vert
B_{1}^{u,i}(\cdot)\right\Vert _{\infty}%
{\displaystyle\int_{t}^{T}}
\left\vert \left(  q_{s}^{u}\right)  ^{i}\right\vert ds-2\sum\limits_{i=1}^{d}%
{\displaystyle\int_{t}^{T}}
\left(  p_{s}^{u},\left(  q_{s}^{u}\right)  ^{i}\right)  dW_{s}^{i}.
\end{array}
\label{q-BMO-1}%
\end{equation}
Note that, by Young's inequality and H\"{o}lder's inequality,
\begin{equation}
\label{q-BMO-2}
\begin{array}
[c]{rl}
& 2\left\Vert p_{{}}^{u}\right\Vert _{\infty}\sum\limits_{i=1}^{d}\left\Vert
B_{1}^{u,i}(\cdot)\right\Vert _{\infty}%
{\displaystyle\int_{t}^{T}}
\left\vert \left(  q_{s}^{u}\right)  ^{i}\right\vert ds\\
\leq & 2T\left\Vert p_{{}}^{u}\right\Vert _{\infty}^{2}\left(  \sum
\limits_{i=1}^{d}\left\Vert B_{1}^{u,i}\right\Vert
_{\infty}\right)  ^{2}+\frac{1}{2T}\left(
{\displaystyle\int_{t}^{T}}
\left\vert \left(  q_{s}^{u}\right)  ^{i}\right\vert ds\right)  ^{2}\\
\leq & 2dT\left\Vert p_{{}}^{u}\right\Vert _{\infty}^{2}\sum\limits_{i=1}%
^{d}\left\Vert B_{1}^{u,i}\right\Vert _{\infty}%
^{2}+\frac{1}{2}%
{\displaystyle\int_{t}^{T}}
\left\vert \left(  q_{s}^{u}\right)  ^{i}\right\vert ^{2}ds
\end{array}
\end{equation}
Consequently, combining (\ref{q-BMO-1}) with (\ref{q-BMO-2}) and taking the conditional expectation on both sides of the inequalities,
we have $\mathbb{E}\left[  \int_{t}^{T}\left\vert q_{s}%
^{u}\right\vert ^{2}ds\mid\mathcal{F}_{t}\right]  \leq C$ for all $t\in
\lbrack0,T]$, which implies that $q_{{}}^{u}$ is bounded in $\mathcal{K}(\mathbb{R}^{n\times d})$ across all $u(\cdot) \in \mathcal{U}[0,T]$.
\end{proof}

Similar to Lemma \ref{adj-1st-lem}, the solution $\left( P^{u}, Q^{u} \right)$ to the second-order
adjoint equation (\ref{2nd-adj-eq}) is uniformly bounded in $L_{\mathcal{F}}^{\infty}([0,T];\mathbb{S}^{n\times n}) \times \left(\mathcal{K}(\mathbb{S}^{n\times n})\right)^{d}$
across all admissible controls.

\begin{lemma}
\label{adj-2nd-lem} Let Assumption \ref{assum-1} hold. Then, for any
$u(\cdot)\in\mathcal{U}[0,T]$, (\ref{2nd-adj-eq}) admits a unique
solution $\left(  P^{u},Q^{u}\right)  \in\mathcal{S}_{\mathcal{F}}%
^{2}([0,T];\mathbb{S}^{n\times n})\times\left(  \mathcal{M}_{{}}%
^{2}(\mathbb{S}^{n\times n})\right)  ^{d}$, where $Q^{u}=\left(  \left(
Q_{{}}^{u}\right)  ^{1},\left(  Q_{{}}^{u}\right)  ^{2},\ldots,\left(  Q_{{}%
}^{u}\right)  ^{d}\right)  $. Moreover, we have
\[
\sup_{u(\cdot)\in\mathcal{U}[0,T]}\left(  \left\Vert P_{{}}^{u}\right\Vert
_{\infty}+\left\Vert Q_{{}}^{u}\right\Vert _{\mathcal{K}}\right)  \leq C,
\]
where $C$ is independent of $u(\cdot)$.
\end{lemma}

\begin{proof}[\textbf{Proof}]
For any given $u(\cdot)\in\mathcal{U}[0,T]$, by Theorem 5.1 in \cite{Karoui97} with the boundedness of $b_{x}$, $\sigma
_{x}$, $\Phi_{x}$, $f_{x}$, $b_{xx}$, $\sigma_{xx}$, $\Phi_{xx}$ and $f_{xx}$,
(\ref{2nd-adj-eq}) admits a unique solution $\left(  P^{u},Q^{u}\right)  \in\mathcal{S}%
_{\mathcal{F}}^{2}([0,T];\mathbb{S}^{n\times n})\times\left(  \mathcal{M}_{{}%
}^{2}(\mathbb{S}^{n\times n})\right)  ^{d}$.
Furthermore, denote by $\left(  P^{u}\right)  ^{j}$ (resp. $\left(  \Psi_{{}%
}^{u}\right)  ^{j}$, $\left(  Q^{u}\right)  ^{ji}$) the $j$th column of
$P^{u}$ (resp. $\Psi_{{}}^{u}$, $\left(  Q^{u}\right)  ^{i}$) for
$i=1,2,\ldots,d$, $j=1,2,\ldots,n$, and $I_{n}^{ij}\in\mathbb{R}_{{}}^{n\times n}$ the matrix whose
elements equal to 0 except that the one in $i$th row and $j$th column
equals to $1$. Set%
\[%
\begin{array}
[c]{l}%
\tilde{P}_{t}^{u}:=\left(  \left(  \left(  P_{t}^{u}\right)  ^{1}\right)
^{\intercal},\left(  \left(  P_{t}^{u}\right)  ^{2}\right)  ^{\intercal
},\ldots,\left(  \left(  P_{t}^{u}\right)  ^{n}\right)  ^{\intercal}\right)
^{\intercal}\in\mathbb{R}_{{}}^{n^{2}},\\
\left(  \tilde{Q}_{s}^{u}\right)  ^{i}:=\left(  \left(  \left(  Q_{t}%
^{u}\right)  ^{1i}\right)  ^{\intercal},\left(  \left(  Q_{t}^{u}\right)
^{2i}\right)  ^{\intercal},\ldots,\left(  \left(  Q_{t}^{u}\right)
^{ni}\right)  ^{\intercal}\right)  ^{\intercal}\in\mathbb{R}_{{}}^{n^{2}},
\end{array}
\]%
\[
A_{2}^{u}(t):=\left(
\begin{array}
[c]{ccc}%
\left(  b_{x}^{u}(t)\right)  ^{\intercal} &  & \\
& \ddots & \\
&  & \left(  b_{x}^{u}(t)\right)  ^{\intercal}%
\end{array}
\right)  \in\mathbb{R}_{{}}^{n^{2}\times n^{2}},\text{ }%
\]
\[
I_{n^{2}}^{\ast}:=\left(
\begin{array}
[c]{ccc}%
I_{n}^{11} & \cdots & I_{n}^{n1}\\
\vdots & \ddots & \vdots\\
I_{n}^{1n} & \cdots & I_{n}^{nn}%
\end{array}
\right)  \in\mathbb{S}_{{}}^{n^{2}\times n^{2}},
\]
\[%
\begin{array}
[c]{l}%
B_{2}^{u,i}(t):=\left(
\begin{array}
[c]{ccc}%
\left(  \sigma_{x}^{u,i}(t)\right)  ^{\intercal} &  & \\
& \ddots & \\
&  & \left(  \sigma_{x}^{u,i}(t)\right)  ^{\intercal}%
\end{array}
\right)  \in\mathbb{R}_{{}}^{n^{2}\times n^{2}},\\
D_{{}}^{u,i}(t):=\left(
\begin{array}
[c]{ccc}%
\left(  \sigma_{x_{1}}^{u,i}(t)\right)  ^{1}I_{n} & \cdots & \left(
\sigma_{x_{1}}^{u,i}(t)\right)  ^{n}I_{n}\\
\vdots & \ddots & \vdots\\
\left(  \sigma_{x_{n}}^{u,i}(t)\right)  ^{1}I_{n} & \cdots & \left(
\sigma_{x_{n}}^{u,i}(t)\right)  ^{n}I_{n}%
\end{array}
\right)  \in\mathbb{R}_{{}}^{n^{2}\times n^{2}},
\end{array}
\]
where $\left( \left( \sigma_{x_{1}}^{u,i}(t)\right)  ^{k},\ldots, \left( \sigma_{x_{n}}^{u,i}(t)\right)  ^{k}\right)$ is the gradient of the $k$th component
of $\sigma^{u,i}(t)$, $k=1,\ldots,n$.
Then (\ref{2nd-adj-eq}) can be rewritten in the vector-valued form:%
\begin{equation}%
\begin{array}
[c]{cl}%
\tilde{P}_{t}^{u}= & \tilde{P}_{T}^{u}+%
{\displaystyle\int_{t}^{T}}
\left\{  \left[  f_{y}^{u}(s)I_{n^{2}}+\sum\limits_{i=1}^{d}f_{z_{i}}%
^{u}(s)\left(  I_{n^{2}}+I_{n^{2}}^{\ast}\right)   B_{2}^{u,i}(s)\right.  \right.  \\
& +\left.  \left(  I_{n^{2}}+I_{n^{2}}^{\ast}\right)  A_{2}^{u}(s)\sum
\limits_{i=1}^{d}  D_{{}}^{u,i}(s)  B_{2}%
^{u,i}(s)\right]  \tilde{P}_{s}^{u}\\
& +\left.  \sum\limits_{i=1}^{d}\left[  f_{z_{i}}^{u}(s)+\left(  I_{n^{2}%
}+I_{n^{2}}^{\ast}\right)   B_{2}^{u,i}(s)\right]  \left(
\tilde{Q}_{s}^{u}\right)  ^{i}+\tilde{\Psi}_{s}^{u}\right\}  ds\\
& -%
\sum\limits_{i=1}^{d}
{\displaystyle\int_{t}^{T}}
\left(  \tilde{Q}_{s}^{u}\right)  ^{i}dW_{s}^{i} \quad \text{ for } t \in [0,T],
\end{array}
\label{2nd-adj-eq-multi-new}%
\end{equation}
where $\tilde{\Psi}_{t}^{u}=\left(  \left(  \left(  \Psi_{t}^{u}\right)  ^{1}\right)
^{\intercal},\left(  \left(  \Psi_{t}^{u}\right)  ^{2}\right)  ^{\intercal
},...,\left(  \left(  \Psi_{t}^{u}\right)  ^{n}\right)  ^{\intercal}\right)
^{\intercal}$ and $I_{n^{2}}$ is the $n^{2}\times n^{2}$ identity matrix.
Obviously, (\ref{2nd-adj-eq-multi-new}) is same as the form of (\ref{1st-adj-eq-linear-form}).
For all $t\in\lbrack0,T]$, recalling (\ref{2nd-adj-eq-multi-data}) and by Lemma \ref{adj-1st-lem}, it can be verified that%
\begin{equation}\label{adj-2st-nonhomo}
\mathbb{E}\left[  \left(  \int_{t}^{T}\left\vert \tilde{\Psi}_{s}%
^{u}\right\vert ds\right)  ^{2}\mid\mathcal{F}_{t}\right]  =\mathbb{E}\left[
\left(  \int_{t}^{T}\left\vert \Psi_{s}^{u}\right\vert ds\right)  ^{2}%
\mid\mathcal{F}_{t}\right]  \leq C,
\end{equation}
where $C$ is independent of $u(\cdot)$.
Indeed, for instance, it follows immediately from the energy inequality and Lemma \ref{adj-1st-lem} that%
\[%
\begin{array}
[c]{l}%
\mathbb{E}\left[  \left(
{\displaystyle\int_{t}^{T}}
\sum\limits_{j,k=1}^{d}\left\vert f_{z_{j}z_{k}}^{u}(s)\left(  q_{s}%
^{u}\right)  ^{j}\left(  \left(  q_{s}^{u}\right)  ^{k}\right)  ^{\intercal
}\right\vert ds\right)  ^{2}\mid\mathcal{F}_{t}\right]  \\
\leq d^{4}\left\Vert f_{zz}\right\Vert _{\infty}^{2}\mathbb{E}\left[  \left(
\int_{t}^{T}\left\vert q_{s}^{u}\right\vert ^{2}dt\right)  ^{2}\mid
\mathcal{F}_{t}\right]  \\
\leq2d^{4}\left\Vert f_{zz}\right\Vert _{\infty}^{2}\left\Vert q_{{}}%
^{u}\right\Vert _{\mathcal{K}}^{2}\\
\leq C.
\end{array}
\]
The other terms in (\ref{2nd-adj-eq-multi-data}) can be estimated similarly so that (\ref{adj-2st-nonhomo}) holds.

Therefore, similar to the proof of Lemma \ref{adj-1st-lem}, due to (\ref{adj-2st-nonhomo}),
one can obtain that $\sup_{u(\cdot)\in\mathcal{U}[0,T]}\left\Vert \tilde{P}^{u}\right\Vert
_{\infty}<\infty$.
This implies that $\sup_{u(\cdot)\in\mathcal{U}%
[0,T]}\left\Vert P^{u}\right\Vert _{\infty}<\infty$.
Then, applying It\^{o}'s formula to $\left \vert \tilde{P}_{s}^{u} \right \vert^{2}$ on $[t,T]$ and
noticing that $\left\vert Q_{s}^{u}\right\vert^{2}=\left\vert \tilde{Q}_{s}^{u}\right\vert^{2}$, we finally obtain the desired result.
\end{proof}

\subsection{Convergence of the Modified MSA}

In order to prove the convergence of Algorithm \ref{algorithm-MSA}, we need
the following lemma about the error estimate. It will be seen that if we
directly minimize $\mathcal{H}$ instead of $\mathcal{\tilde{H}}$ (Step 5 in
Algorithm \ref{algorithm-MSA}), then the updated control variable may fail to
make the cost functional descend efficiently.

For any given $u(\cdot)$, $v(\cdot)\in\mathcal{U}[0,T]$,
define a new probability $\mathbb{\tilde{P}}$ and a Brownian motion $\tilde
{W}$ with respect to $\mathbb{\tilde{P}}$ by
\begin{equation}
d\mathbb{\tilde{P}}:=\mathcal{E}\left(  \sum\limits_{i=1}^{d}\int_{0}%
^{T}f_{z_{i}}^{u}(t)dW_{t}^{i}\right)  d\mathbb{P};\text{ \ }\tilde{W}%
_{t}^{i}:=W_{t}^{i}-\int_{0}^{t}f_{z_{i}}^{u}(s)ds, \text{ } i=1,2,\ldots,d.
\label{new-P-W}%
\end{equation}
Denote by $\mathbb{\tilde{E}}\left[  \cdot\right]  $ the mathematical
expectation corresponding to $\mathbb{\tilde{P}}$.

\begin{lemma}
\label{lem-costfunc-dominate} Let Assumption \ref{assum-1} hold. Then there exists a
universal constant $C>0$ such that%
\begin{equation}%
\begin{array}
[c]{l}%
J\left(  v(\cdot)\right)  -J\left(  u(\cdot)\right)  \\
\leq e^{\left\Vert f_{y}\right\Vert _{\infty}T}\mathbb{\tilde{E}}\left[
\int_{0}^{T}\mathcal{\hat{H}}^{+}\mathcal{(}t\mathcal{)}dt\right]
-e^{-\left\Vert f_{y}\right\Vert _{\infty}T}\mathbb{\tilde{E}}\left[  \int%
_{0}^{T}\mathcal{\hat{H}}^{-}\mathcal{(}t\mathcal{)}dt\right]  \\
\ \ +\ C\left\{  \sum\limits_{\psi\in\{b,\sigma\}}\mathbb{\tilde{E}}\left[
\int_{0}^{T}\left\vert \psi(t,X_{t}^{u},v_{t})-\psi(t,X_{t}^{u},u_{t}%
)\right\vert ^{2}dt\right]  \right.  \\
\ \ +\mathbb{\tilde{E}}\left[  \int_{0}^{T}\left\vert f(t,\Theta_{t}^{u}%
,v_{t})-f(t,\Theta_{t}^{u},u_{t})\right\vert ^{2}dt\right]  \\
\ \ +\left.  \sum\limits_{w\in\{x,y,z\}}\mathbb{\tilde{E}}\left[  \int_{0}%
^{T}\left\vert G_{w}(t,\Theta_{t}^{u},p_{t}^{u},q_{t}^{u},v_{t},u_{t}%
)-G_{w}(t,\Theta_{t}^{u},p_{t}^{u},q_{t}^{u},u_{t},u_{t})\right\vert
^{2}dt\right]  \right\}  ,
\end{array}
\label{est-Jv-Ju}%
\end{equation}
where $\Theta_{t}^{u}$ is defined in (\ref{def-notation}),
\[
\mathcal{\hat{H}(}t\mathcal{)}=\mathcal{H}(t,\Theta_{t}^{u},%
p_{t}^{u},q_{t}^{u},P_{t}^{u},v_{t},u_{t})-\mathcal{H}(t,\Theta_{t}^{u},%
p_{t}^{u},q_{t}^{u},P_{t}^{u},u_{t},u_{t}),\text{    } t\in[0,T]
\]
and $\hat{\mathcal{H}}^{+}(\cdot)$, $\hat{\mathcal{H}}^{-}(\cdot)$ are respectively the positive, negative part of $\hat{\mathcal{H}}(\cdot)$.
\end{lemma}

\begin{proof}[\textbf{Proof}]
Let $u(\cdot)$, $v(\cdot)\in\mathcal{U}[0,T]$ be given.
Denote by
\[
\eta(t)=Y_{t}^{v}-Y_{t}^{u}-\left(  p_{t}^{u}\right)  ^{\intercal}(X_{t}%
^{v}-X_{t}^{u})-\frac{1}{2}\mathrm{tr}\left\{  P_{t}^{u}(X_{t}^{v}-X_{t}%
^{u})\left(  X_{t}^{v}-X_{t}^{u}\right)  ^{\intercal}\right\}
\]
and, for $i=1,\ldots,d$,
\begin{equation}%
\begin{array}
[c]{rl}%
\zeta^{i}(t)= & \left(  Z_{t}^{v}\right)  ^{i}-\left(  Z_{t}^{u}\right)
^{i}-\tilde{\Delta}^{i}(t,X_{t}^{u},p_{t}^{u},v_{t},u_{t})\\
& -\left(  \Upsilon^{i}(t,X_{t}^{u},p_{t}^{u},q_{t}^{u},u_{t})\right)
^{\intercal}(X_{t}^{v}-X_{t}^{u})-R_{1}^{i}(t)\\
& -\left(  p_{t}^{u}\right)  ^{\intercal}\left[  \sigma_{x}^{i}(t,X_{t}%
^{u},v_{t})-\sigma_{x}^{u,i}(t)\right]  (X_{t}^{v}-X_{t}%
^{u})\text{ \ }\\
& -\frac{1}{2}\sum\limits_{j=1}^{n}\left(  p_{t}^{u}\right)  ^{j}%
\mathrm{tr}\left\{  \left(  \sigma_{xx}^{u,i}(t)\right)  ^{j}\left(  X_{t}%
^{v}-X_{t}^{u}\right)  \left(  X_{t}^{v}-X_{t}^{u}\right)  ^{\intercal
}\right\}  -\frac{1}{2}R_{2}^{i}(t)\\
& -\frac{1}{2}\mathrm{tr}\left\{  \left[  \left(  \sigma_{x}^{u,i}(t)\right)
^{\intercal}P_{t}^{u}+\left(  P_{t}^{u}\right)  ^{\intercal}\sigma_{x}%
^{u,i}(t)+Q_{t}^{u,i}\right]  (X_{t}^{v}-X_{t}^{u})\left(  X_{t}^{v}-X_{t}%
^{u}\right)  ^{\intercal}\right\}  ,
\end{array}
\label{zeta}%
\end{equation}
where
\[%
\begin{array}
[c]{rl}
& R_{1}^{i}(t)\\
= & \sum\limits_{j=1}^{n}\left(  p_{t}^{u}\right)  ^{j}\cdot\int_{0}^{1}%
\int_{0}^{1}\\
& \lambda\left(  X_{t}^{v}-X_{t}^{u}\right)  ^{\intercal}\left[  \sigma
_{xx}^{ji}(t,X_{t}^{u}+\lambda\mu(X_{t}^{v}-X_{t}^{u}),v_{t})-\left(
\sigma_{xx}^{u,i}(t)\right)  ^{j}\right]  (X_{t}^{v}-X_{t}^{u})d\lambda d\mu;
\end{array}
\]
\[%
\begin{array}
[c]{rl}%
R_{2}^{i}(t)= & \left(  X_{t}^{v}-X_{t}^{u}\right)  ^{\intercal}P_{t}^{u}%
\Pi_{1}^{i}(t)+\left(  \Pi_{1}^{i}(t)\right)  ^{\intercal}\left(  P_{t}%
^{u}\right)  ^{\intercal}(X_{t}^{v}-X_{t}^{u});
\end{array}
\]%
\[%
\begin{array}
[c]{rl}%
\Pi_{1}^{i}(t)= & \int_{0}^{1}\left[  \sigma_{x}^{u,i}(t,X_{t}^{u}%
+\lambda(X_{t}^{v}-X_{t}^{u}),v_{t})-\sigma_{x}^{u,i}(t)\right]  d\lambda
\cdot(X_{t}^{v}-X_{t}^{u})\\
& +\sigma^{i}(t,X_{t}^{u},v_{t})-\sigma^{u,i}(t).
\end{array}
\]
Denote $\tilde{D}^{2}f^{u,v}(t)=2\int_{0}^{1}\int_{0}^{1}\lambda D^{2}f(t,\Theta_{t}^{u,v}
+\lambda\mu(\Theta_{t}^{v}-\Theta_{t}^{u,v}),v_{t})d\mu d\lambda$ and note that%
\[%
\begin{array}
[c]{rl}%
G_{x}(t,x,y,z,p,q,v,u)= & b_{x}^{\intercal}(t,x,v)p+\sum\limits_{i=1}%
^{d}\left(  \sigma_{x}^{i}(t,x,v)\right)  ^{\intercal}q^{i}\\
& +f_{x}(t,x,y,z+\tilde{\Delta}(t,x,p,v,u),v)\\
& +\tilde{\Delta}_{x}^{\intercal}(t,x,p,v,u)f_{z}(t,x,y,z+\tilde{\Delta
}(t,x,p,v,u),v),\\
G_{y}(t,x,y,z,p,q,v,u)= & f_{y}(t,x,y,z+\tilde{\Delta}(t,x,p,v,u),v),\\
G_{z}(t,x,y,z,p,q,v,u)= & f_{z}(t,x,y,z+\tilde{\Delta}(t,x,p,v,u),v).
\end{array}
\]
Then, applying It\^{o}'s formula to $\eta$ on $[t,T]$, we have%
\begin{equation}%
\begin{array}
[c]{rl}%
\eta(t)= & \eta(T)+%
{\displaystyle\int_{t}^{T}}
\left\{  f_{y}^{u}(s)\eta(s)+\sum\limits_{i=1}^{d}f_{z_{i}}^{u}(s)\zeta
^{i}(s)+\left[  f_{y}^{u,v}(s)-f_{y}^{u}(s)\right]  (Y_{s}^{v}-Y_{s}%
^{u})\right.  \\
& +\mathcal{H}(s,\Theta_{s}^{u},p_{s}^{u},q_{s}^{u},P_{s}^{u},v_{s}%
,u_{s})-\mathcal{H}(s,\Theta_{s}^{u},p_{s}^{u},q_{s}^{u},P_{s}^{u},u_{s}%
,u_{s})\\
& +\left(  G_{x}(s,\Theta_{s}^{u},p_{s}^{u},q_{s}^{u},v_{s},u_{s}%
)-G_{x}(s,\Theta_{s}^{u},p_{s}^{u},q_{s}^{u},u_{s},u_{s})\right)  ^{\intercal
}(X_{s}^{v}-X_{s}^{u})\\
& +\sum\limits_{i=1}^{d}\left[  f_{z_{i}}^{u,v}(s)-f_{z_{i}}^{u}(s)\right]
\cdot\left[  \left(  Z_{s}^{v}\right)  ^{i}-\left(  Z_{s}^{u}\right)
^{i}-\tilde{\Delta}^{i}(s,X_{s}^{u},p_{s}^{u},v_{s},u_{s})\right.  \\
& -\left.  \left(  p_{s}^{u}\right)  ^{\intercal}\left(  \sigma_{x}%
^{i}(s,X_{s}^{u},v_{s})-\sigma_{x}^{u,i}(s)\right)  (X_{s}%
^{v}-X_{s}^{u})\right]  \\
& +\left.  \sum\limits_{i=1}^{d}f_{z_{i}}^{u}(s)\left(  R_{1}^{i}(s)+\frac
{1}{2}R_{2}^{i}(s)\right)  +R_{3}(s)+R_{4}(s)+R_{5}(s)+\frac{1}{2}%
R_{6}(s)\right\}  ds\\
& -\sum\limits_{i=1}^{d}%
{\displaystyle\int_{t}^{T}}
\zeta^{i}(s)dW_{s}^{i},
\end{array}
\label{eta-1}%
\end{equation}
where%
\[
\eta(T)=\int_{0}^{1}\int_{0}^{1}\lambda\left(  X_{T}^{v}-X_{T}^{u}\right)
^{\intercal}\left[  \Phi_{xx}(X_{T}^{u}+\lambda\mu(X_{T}^{v}-X_{T}^{u}%
))-\Phi_{xx}(X_{T}^{u})\right]  (X_{t}^{v}-X_{t}^{u})d\lambda d\mu;
\]%
\[%
\begin{array}
[c]{l}%
\ \ \ R_{3}(t)\\
=\sum\limits_{i=1}^{d}\sum\limits_{j=1}^{n}\left(  q_{t}^{u}\right)
^{ji}\cdot\int_{0}^{1}\int_{0}^{1}\\
\lambda\left(  X_{t}^{v}-X_{t}^{u}\right)  ^{\intercal}\left[  \sigma
_{xx}^{ji}(t,X_{t}^{u}+\lambda\mu(X_{t}^{v}-X_{t}^{u}),v_{t})-\left(
\sigma_{xx}^{u,i}(t)\right)  ^{j}\right]  (X_{t}^{v}-X_{t}^{u})d\lambda d\mu;
\end{array}
\]
\[%
\begin{array}
[c]{rl}
& R_{4}(t)\\
= & \dfrac{1}{2}\left[  \left(  X_{t}^{v}-X_{t}^{u}\right)  ^{\intercal}%
,Y_{t}^{v}-Y_{t}^{u},\left(  \left(  Z_{t}^{v}\right)  -\left(  Z_{t}%
^{u}\right)  -\tilde{\Delta}(t,X_{t}^{u},p_{t}^{u},v_{t},u_{t})\right)
^{\intercal}\right]  \\
& \cdot\left[  \tilde{D}^{2}f^{u,v}(t)-D^{2}f^{u}(t)\right]  \\
& \cdot\left[  \left(  X_{t}^{v}-X_{t}^{u}\right)  ^{\intercal},Y_{t}%
^{v}-Y_{t}^{u},\left(  \left(  Z_{t}^{v}\right)  -\left(  Z_{t}^{u}\right)
-\tilde{\Delta}(t,X_{t}^{u},p_{t}^{u},v_{t},u_{t})\right)  ^{\intercal
}\right]  ^{\intercal},\\
& +\dfrac{1}{2}\left[  \mathbf{0}_{n\times1}^{\intercal},Y_{t}^{v}-Y_{t}%
^{u}-\left(  X_{t}^{v}-X_{t}^{u}\right)  ^{\intercal}p_{t}^{u},\right.  \\
& \left.  \left(  \left(  Z_{t}^{v}\right)  -\left(  Z_{t}^{u}\right)
-\tilde{\Delta}(t,X_{t}^{u},p_{t}^{u},v_{t},u_{t})-\Upsilon^{\intercal
}(t,X_{t}^{u},p_{t}^{u},q_{t}^{u},u_{t})\left(  X_{t}^{v}-X_{t}^{u}\right)
\right)  ^{\intercal}\right]  \\
& \cdot D^{2}f^{u}(t)\cdot\left[  \mathbf{0}_{n\times1}^{\intercal},Y_{t}%
^{v}-Y_{t}^{u}-\left(  X_{t}^{v}-X_{t}^{u}\right)  ^{\intercal}p_{t}%
^{u},\right.  \\
& \left.  \left(  \left(  Z_{t}^{v}\right)  -\left(  Z_{t}^{u}\right)
-\tilde{\Delta}(t,X_{t}^{u},p_{t}^{u},v_{t},u_{t})-\Upsilon^{\intercal
}(t,X_{t}^{u},p_{t}^{u},q_{t}^{u},u_{t})\left(  X_{t}^{v}-X_{t}^{u}\right)
\right)  ^{\intercal}\right]  ^{\intercal},
\end{array}
\]
where $\mathbf{0}_{n\times 1}= ( \overbrace{ 0,\ldots,0 }^{n} )^{\intercal}$;
\begin{align*}
&  R_{5}(t)=\sum_{j=1}^{n}\left(  p_{t}^{u}\right)  ^{j}\\
&  \cdot\int_{0}^{1}\int_{0}^{1}\lambda\left(  X_{t}^{v}-X_{t}^{u}\right)
^{\intercal}\left[  b_{xx}^{j}(t,X_{t}^{u}+\lambda\mu(X_{t}^{v}-X_{t}%
^{u}),v_{t})-b_{xx}^{u,j}(t)\right]  (X_{t}^{v}-X_{t}^{u})d\lambda d\mu;
\end{align*}
\[%
\begin{array}
[c]{rl}
& R_{6}(t)\\
= & \left(  X_{t}^{v}-X_{t}^{u}\right)  ^{\intercal}P_{t}^{u}\Pi
_{2}(t)+\left(  \Pi_{2}(t)\right)  ^{\intercal}\left(  P_{t}^{u}\right)
^{\intercal}(X_{t}^{v}-X_{t}^{u})\\
& +\sum\limits_{i=1}^{d}\left[  \left(  X_{t}^{v}-X_{t}^{u}\right)
^{\intercal}Q_{t}^{u,i}\Pi_{1}^{i}(t)+\left(  \Pi_{1}^{i}(t)\right)
^{\intercal}\left(  Q_{t}^{u,i}\right)  ^{\intercal}(X_{t}^{v}-X_{t}%
^{u})\right]  \\
& +\sum\limits_{i=1}^{d}\left[  \left(  X_{t}^{v}-X_{t}^{u}\right)
^{\intercal}\left(  \sigma_{x}^{u,i}(t)\right)  ^{\intercal}P_{t}^{u}\Pi
_{1}^{i}(t)+\left(  \Pi_{1}^{i}(t)\right)  ^{\intercal}\left(  P_{t}%
^{u}\right)  ^{\intercal}\sigma_{x}^{u,i}(t)(X_{t}^{v}-X_{t}^{u})\right]  \\
& +\sum\limits_{i=1}^{d}\left(  \Pi_{1}^{i}(t)-\sigma^{i}(t,X_{t}^{u}%
,v_{t})+\sigma^{u,i}(t)\right)  ^{\intercal}P_{t}^{u}\left(  \Pi_{1}%
^{i}(t)-\sigma^{i}(t,X_{t}^{u},v_{t})+\sigma^{u,i}(t)\right)
\end{array}
\]
with
\[
\Pi_{2}(t)=\int_{0}^{1}\left[  b_{x}^{u}(t,X_{t}^{u}+\lambda(X_{t}^{v}%
-X_{t}^{u}),v_{t})-b_{x}^{u}(t)\right]  d\lambda\cdot(X_{t}^{v}-X_{t}%
^{u})+b(t,X_{t}^{u},v_{t})-b^{u}(t).
\]
From the definition of $\hat{\mathcal{H}}$, (\ref{eta-1}) can be rewritten
as
\begin{equation}%
\begin{array}
[c]{rl}%
\eta(t)= & \eta(T)+%
{\displaystyle\int_{t}^{T}}
\left\{  f_{y}^{u}(s)\eta(s)+\sum\limits_{i=1}^{d}f_{z_{i}}^{u}(s)\zeta
^{i}(s)+\mathcal{\hat{H}(}s\mathcal{)}+\phi_{s}\right\}  \\
& -\sum\limits_{i=1}^{d}%
{\displaystyle\int_{t}^{T}}
\zeta^{i}(s)dW_{s}^{i},
\end{array}
\label{eta-2}%
\end{equation}
where%
\[%
\begin{array}
[c]{rl}%
\phi_{t}= & \left(  G_{x}(t,\Theta_{t}^{u},p_{t}^{u},q_{t}^{u},v_{t}%
,u_{t})-G_{x}(t,\Theta_{t}^{u},p_{t}^{u},q_{t}^{u},u_{t},u_{t})\right)
^{\intercal}(X_{t}^{v}-X_{t}^{u})\\
& +\left[  f_{y}^{u,v}(t)-f_{y}^{u}(t)\right]  (Y_{t}^{v}-Y_{t}^{u})\\
& +\sum\limits_{i=1}^{d}\left[  f_{z_{i}}^{u,v}(t)-f_{z_{i}}^{u}(t)\right]
\cdot\left[  \left(  Z_{t}^{v}\right)  ^{i}-\left(  Z_{t}^{u}\right)
^{i}-\tilde{\Delta}^{i}(t,X_{t}^{u},p_{t}^{u},v_{t},u_{t})\right.  \\
& -\left.  \left(  p_{t}^{u}\right)  ^{\intercal}\left[  \sigma_{x}%
^{i}(t,X_{t}^{u},v_{t})-\sigma_{x}^{u,i}(t)\right]  (X_{t}%
^{v}-X_{t}^{u})\right]  \\
& +\sum\limits_{i=1}^{d}f_{z_{i}}^{u}(t)\left(  R_{1}^{i}(t)+\frac{1}{2}%
R_{2}^{i}(t)\right)  +R_{3}(t)+R_{4}(t)+R_{5}(t)+\frac{1}{2}R_{6}(t).
\end{array}
\]
Due to (\ref{new-P-W}), (\ref{eta-2}) can be further rewritten as%
\begin{equation}
\eta(t)=\eta(T)+\int_{t}^{T}\left[  f_{y}^{u}(s)\eta(s)+\mathcal{\hat{H}%
(}s\mathcal{)}+\phi_{s}\right]  ds-\sum\limits_{i=1}^{d}\int_{t}^{T}\zeta
^{i}(s)d\tilde{W}_{s}^{i}.\label{eta-3}%
\end{equation}
Applying It\^{o}'s formula to $\exp\left\{  \int_{0}^{\cdot}f_{y}%
^{u}(s)ds\right\}  \eta$ on $[t,T]$, we get%
\begin{equation}%
\begin{array}
[c]{rl}%
\eta(t)= & \exp\left\{  \int_{t}^{T}f_{y}^{u}(s)ds\right\}  \eta(T)+\int%
_{t}^{T}\exp\left\{  \int_{t}^{s}f_{y}^{u}(r)dr\right\}  \left(
\mathcal{\hat{H}(}s\mathcal{)}+\phi_{s}\right)  ds\\
& -\sum\limits_{i=1}^{d}\int_{t}^{T}\exp\left\{  \int_{t}^{s}f_{y}%
^{u}(r)dr\right\}  \zeta^{i}(s)d\tilde{W}_{s}^{i}.
\end{array}
\label{eta-explicit}%
\end{equation}
One can check $\mathbb{\tilde{E}}\left[  \left(  \int_{0}%
^{T}\left\vert \zeta(t)\right\vert ^{2}dt\right)  ^{\frac{1}{2}}\right]
<\infty$ so the stochastic integral in (\ref{eta-explicit}) is a true martingale
under $\mathbb{\tilde{P}}$. Then, by taking $\mathbb{\tilde{E}} \left [ \cdot \right ]$
on both sides of (\ref{eta-explicit}), we obtain
\begin{equation}%
\begin{array}
[c]{rl}%
\eta(0)= & J\left(  v(\cdot)\right)  -J\left(  u(\cdot)\right)  \\
= & \mathbb{\tilde{E}}\left[  \exp\left\{  \int_{0}^{T}f_{y}^{u}(t)dt\right\}
\eta(T)+\int_{0}^{T}\exp\left\{  \int_{0}^{t}f_{y}^{u}(s)ds\right\}  \left(
\mathcal{\hat{H}(}t\mathcal{)}+\phi_{t}\right)  dt\right]  \\
\leq & \exp\left\{  \left\Vert f_{y}\right\Vert _{\infty}T\right\}
\mathbb{\tilde{E}}\left[  \left\vert \eta(T)\right\vert +\int_{0}%
^{T}\left\vert \phi_{t}\right\vert dt\right]  \\
& +\exp\left\{  \left\Vert f_{y}\right\Vert _{\infty}T\right\}  \mathbb{\tilde
{E}}\left[  \int_{0}^{T}\mathcal{\hat{H}}^{+}\mathcal{(}t\mathcal{)}dt\right]
-\exp\left\{  -\left\Vert f_{y}\right\Vert _{\infty}T\right\}  \mathbb{\tilde
{E}}\left[  \int_{0}^{T}\mathcal{\hat{H}}^{-}\mathcal{(}t\mathcal{)}dt\right]
.
\end{array}
\label{Jv-Ju-origin}%
\end{equation}
Thus, in order to obtain (\ref{est-Jv-Ju}), we shall proceed to estimate $\mathbb{\tilde{E}%
}\left[  \left\vert \eta(T)\right\vert +\int_{0}^{T}\left\vert \phi
_{t}\right\vert dt\right]  $ as the following three parts.

\noindent \textbf{(i)} \textit{Estimate of}\textbf{ }$\mathbb{\tilde{E}}\left[  \left\vert \eta(T)\right\vert +\int_{0}^{T}\left(
\sum_{i=1}^{d}\left\vert f_{z_{i}}^{u}(t)\left(  R_{1}^{i}(t)+\frac{1}%
{2}R_{2}^{i}(t)\right)  \right\vert +\left\vert R_{5}(t)\right\vert \right)
dt\right]  $

Denote by
\[%
\begin{array}
[c]{l}%
\tilde{b}_{x}^{u,v}(t)=\int_{0}^{1}b_{x}(t,X_{t}^{u}+\lambda(X_{t}^{v}%
-X_{t}^{u}),v_{t})d\lambda;\\
\tilde{f}_{x}^{u,v}(t)=\int_{0}^{1}f_{x}(t,\Theta_{t}^{u}+\lambda(\Theta
_{t}^{v}-\Theta_{t}^{u}),v_{t})d\lambda.
\end{array}
\]
$\tilde{\sigma}_{x}^{u,v}(t)$, $\tilde{f}_{y}^{u,v}(t)$ and $\tilde{f}%
_{z}^{u,v}(t)$ are defined similarly. Since%
\[%
\begin{array}
[c]{rl}
& X_{t}^{v}-X_{t}^{u}\\
= &
{\displaystyle\int_{0}^{t}}
\left\{  \left(  \tilde{b}_{x}^{u,v}(s)+\sum\limits_{i=1}^{d}f_{z_{i}}%
^{u}(s)\left(  \tilde{\sigma}_{x}^{u,v}(s)\right)  ^{i}\right)  (X_{s}%
^{v}-X_{s}^{u})+b(s,X_{s}^{u},v_{s})-b^{u}(s)\right.  \\
& +\left.  \sum\limits_{i=1}^{d}f_{z_{i}}^{u}(s)\left[  \sigma^{i}(s,X_{s}%
^{u},v_{s})-\sigma^{u,i}(s)\right]  \right\}  ds\\
& +\sum\limits_{i=1}^{d}%
{\displaystyle\int_{0}^{t}}
\left\{  \left(  \tilde{\sigma}_{x}^{u,v}(s)\right)  ^{i}(X_{s}^{v}-X_{s}%
^{u})+\left[  \sigma^{i}(s,X_{s}^{u},v_{s})-\sigma^{u,i}(s)\right]  \right\}
d\tilde{W}_{s}^{i},
\end{array}
\]
by using a standard SDE estimate, we obtain%
\begin{align}
& \mathbb{\tilde{E}}\left[  \sup\limits_{t\in\lbrack0,T]}\left\vert X_{t}%
^{v}-X_{t}^{u}\right\vert ^{2}\right]  \label{est-Xv-Xu}\\
& \leq C\mathbb{\tilde{E}}\left[  \int_{0}^{T}\left\vert b(t,X_{t}^{u}%
,v_{t})-b(t,X_{t}^{u},u_{t})\right\vert ^{2}dt+\int_{0}^{T}\left\vert
\sigma(t,X_{t}^{u},v_{t})-\sigma(t,X_{t}^{u},u_{t})\right\vert ^{2}dt\right]
.\nonumber
\end{align}
Notice that, for $i=1,2,\ldots,d$,%
\[%
\begin{array}
[c]{rl}
& \left\vert \left(  X_{t}^{v}-X_{t}^{u}\right)  ^{\intercal}P_{t}^{u}\Pi
_{1}^{i}(t)\right\vert \\
\leq & \left\Vert P^{u}\right\Vert _{\infty}\left\vert X_{t}^{v}-X_{t}%
^{u}\right\vert \cdot\left(  2\left\Vert \sigma_{x}^{u,i}\right\Vert _{\infty
}\left\vert X_{t}^{v}-X_{t}^{u}\right\vert +\left\vert \sigma^{i}(t,X_{t}%
^{u},v_{t})-\sigma^{i}(t,X_{t}^{u},u_{t})\right\vert \right)  \\
\leq & \left\Vert P^{u}\right\Vert _{\infty}\left(  2\left\Vert \sigma_{x}%
^{u}\right\Vert _{\infty}+\frac{1}{2}\right)  \left\vert X_{t}^{v}-X_{t}%
^{u}\right\vert ^{2}+\frac{1}{2}\left\Vert P^{u}\right\Vert _{\infty
}\left\vert \sigma(t,X_{t}^{u},v_{t})-\sigma(t,X_{t}^{u},u_{t})\right\vert
^{2}.
\end{array}
\]
Therefore, applying Lemmas \ref{adj-1st-lem} and \ref{adj-2nd-lem}, (\ref{est-Xv-Xu}) together with the boundedness of $\Phi_{xx}$,
$b_{xx}$, $\sigma_{xx}$ and $f_{z}$ implies that
\[%
\begin{array}
[c]{rl}
& \mathbb{\tilde{E}}\left[  \left\vert \eta(T)\right\vert +\sum\limits_{i=1}%
^{d}\int_{0}^{T}\left\vert f_{z_{i}}^{u}(t)\left(  R_{1}^{i}(t)+\frac{1}%
{2}R_{2}^{i}(t)\right)  \right\vert dt+\int_{0}^{T}\left\vert R_{5}%
(t)\right\vert dt\right]  \\
\leq & C\mathbb{\tilde{E}}\left[  \int_{0}^{T}\left\vert b(t,X_{t}^{u}%
,v_{t})-b(t,X_{t}^{u},u_{t})\right\vert ^{2}dt+\int_{0}^{T}\left\vert
\sigma(t,X_{t}^{u},v_{t})-\sigma(t,X_{t}^{u},u_{t})\right\vert ^{2}dt\right].
\end{array}
\]

\noindent \textbf{(ii)} \textit{Estimate of}\textbf{ }$\mathbb{\tilde{E}}\left[
\int_{0}^{T} \left\vert R_{3}(t)\right\vert  dt\right]  $

By Lemma \ref{adj-1st-lem}, since $\sup_{u(\cdot)\in\mathcal{U}[0,T]}\left\Vert q_{{}}^{u}\right\Vert
_{\mathcal{K}}<\infty$, then, for $i=1,2,\ldots,d$ and $j=1,2,\ldots,n$,
we have $\left(  q_{{}}^{u}\right)  ^{ji}\cdot W^{i}\in \mathrm{BMO}$.
Furthermore, by Lemma \ref{lem-equiv-BMO-norm},
\begin{equation}
c_{1}\left\Vert \left(  q_{{}}^{u}\right)  ^{ji}\cdot W^{i}\right\Vert
_{\mathrm{BMO}}\leq\left\Vert \left(  q_{{}}^{u}\right)  ^{ji}\cdot\tilde{W}%
^{i}\right\Vert _{\mathrm{BMO}(\mathbb{\tilde{P}})}\leq c_{2}\left\Vert \left(  q_{{}%
}^{u}\right)  ^{ji}\cdot W^{i}\right\Vert _{\mathrm{BMO}},\label{equiv-measure-BMO}%
\end{equation}
where $c_{1}$ and $c_{2}$ are two constants depending only on $\left\Vert
f_{z}\right\Vert _{\infty}$ and $T$. Then it follows from Fefferman's inequality,
the estimate (\ref{est-Xv-Xu}) and the inequality (\ref{equiv-measure-BMO}) that
\[%
\begin{array}
[c]{l}%
\mathbb{\tilde{E}}\left[
{\displaystyle\int_{0}^{T}}
\left\vert R_{3}(t)\right\vert dt\right]  \\
\leq 2\left\Vert \sigma_{xx}\right\Vert _{\infty}\sum\limits_{i=1}^{d}%
\sum\limits_{j=1}^{n}\mathbb{\tilde{E}}\left[
{\displaystyle\int_{0}^{T}}
\left\vert \left(  q_{t}^{u}\right)  ^{ji}\right\vert \left\vert X_{t}%
^{v}-X_{t}^{u}\right\vert ^{2}dt\right]  \\
\leq 2\left\Vert \sigma_{xx}\right\Vert _{\infty}\sum\limits_{i=1}^{d}%
\sum\limits_{j=1}^{n}\mathbb{\tilde{E}}\left[
{\displaystyle\int_{0}^{T}}
\left\vert X_{t}^{v}-X_{t}^{u}\right\vert ^{2}\cdot\left\vert d\left\langle
\left\vert \left(  q_{t}^{u}\right)  ^{ji}\right\vert \cdot\tilde{W}%
^{i},\tilde{W}^{i}\right\rangle _{t}\right\vert \right]  \\
\leq 2\sqrt{2T}\left\Vert \sigma_{xx}\right\Vert _{\infty}\sum\limits_{i=1}%
^{d}\sum\limits_{j=1}^{n}\left\Vert \left(  q_{{}}^{u}\right)  ^{ji}%
\cdot\tilde{W}^{i}\right\Vert _{\mathrm{BMO}(\mathbb{\tilde{P}})}\mathbb{\tilde{E}%
}\left[  \sup\limits_{t\in\lbrack0,T]}\left\vert X_{t}^{v}-X_{t}%
^{u}\right\vert ^{2}\right]  \\
\leq 2\sqrt{2T}ndc_{2}\left\Vert \sigma_{xx}\right\Vert _{\infty}\left\Vert
q_{{}}^{u}\right\Vert _{\mathcal{K}}\mathbb{\tilde{E}}\left[  \sup
\limits_{t\in\lbrack0,T]}\left\vert X_{t}^{v}-X_{t}^{u}\right\vert
^{2}\right]  \\
\leq C\mathbb{\tilde{E}}\left[
{\displaystyle\int_{0}^{T}}
\left\vert b(t,X_{t}^{u},v_{t})-b(t,X_{t}^{u},u_{t})\right\vert ^{2}dt+%
{\displaystyle\int_{0}^{T}}
\left\vert \sigma(t,X_{t}^{u},v_{t})-\sigma(t,X_{t}^{u},u_{t})\right\vert
^{2}dt\right]  .
\end{array}
\]

\noindent \textbf{(iii)} \textit{Estimate of}\textbf{ }$\mathbb{\tilde{E}}\left[
\int_{0}^{T}\left\vert R_{4}(t)\right\vert dt\right]  $

In order to estimate $\mathbb{\tilde{E}}\left[  \int_{0}^{T}\left\vert
R_{4}(t)\right\vert dt\right]  $, we only need to estimate the following two terms:%
\begin{equation}%
\begin{array}
[c]{l}%
\mathbb{\tilde{E}}\left[  \int_{0}^{T}\left(  \int_{0}^{1}\int_{0}^{1}%
\lambda\left\vert f_{z_{i}z_{i}}(t,\Theta_{t}^{u,v}+\lambda\mu(\Theta_{t}%
^{v}-\Theta_{t}^{u,v}),v_{t})- f_{z_{i}z_{i}}^{u}(t)\right\vert d\mu d\lambda\right)  \right.  \\
\ \cdot\left.  \left\vert \left(  Z_{t}^{v}\right)  ^{i}-\left(  Z_{t}%
^{u}\right)  ^{i}-\tilde{\Delta}^{i}(t,X_{t}^{u},p_{t}^{u},v_{t}%
,u_{t})\right\vert ^{2}dt\right]
\end{array}
\label{est-fzz-Z-delata-square}%
\end{equation}
for any $i\in\left\{  1,2,\ldots d\right\}  $, and
\begin{equation}
\label{est-fzz-q-square-X-square}
\mathbb{\tilde{E}}\left[  \int_{0}^{T}\left\vert \sum\limits_{i,j=1}%
^{d}f_{z_{i}z_{j}}^{u}(t)\left(  X_{t}^{v}-X_{t}^{u}\right)  ^{\intercal}\left(
q_{t}^{u}\right)  ^{i}\left(  \left(  q_{t}^{u}\right)  ^{j}\right)
^{\intercal}(X_{t}^{v}-X_{t}^{u})\right\vert dt\right].
\end{equation}

On the one hand, since%
\[%
\begin{array}
[c]{rl}%
Y_{t}^{v}-Y_{t}^{u}= & \Phi(X_{T}^{v})-\Phi(X_{T}^{u})+\int_{t}^{T}\left\{
\left(  \tilde{f}_{x}^{u,v}(s)\right)  ^{\intercal}(X_{s}^{v}-X_{s}%
^{u})\right.  \\
& +\tilde{f}_{y}^{u,v}(s)(Y_{s}^{v}-Y_{s}^{u})+\left(  \tilde{f}_{z}%
^{u,v}(s)-f_{z}^{u}(s)\right)  ^{\intercal}(Z_{s}^{v}-Z_{s}^{u})\\
& -\left.  \left(  \tilde{f}_{z}^{u,v}(s)\right)  ^{\intercal}\tilde{\Delta
}(s,X_{s}^{u},p_{s}^{u},v_{s},u_{s})+f^{u,v}(s)-f^{u}(s)\right\}  ds\\
& -\int_{t}^{T}(Z_{s}^{v}-Z_{s}^{u})^{\intercal}d\tilde{W}_{s}^{i}
\end{array}
\]
and
\[%
\begin{array}
[c]{rl}
& \left\vert f^{u,v}(t)-f^{u}(t)\right\vert \\
= & \left\vert f(t,X_{t}^{u},Y_{t}^{u},Z_{t}^{u}+\tilde{\Delta}(t,X_{t}%
^{u},p_{t}^{u},v_{t},u_{t}),v_{t})-f(t,X_{t}^{u},Y_{t}^{u},Z_{t}^{u}%
,u_{t})\right\vert \\
\leq & \left\Vert f_{z}\right\Vert _{\infty}\left\Vert p^{u}\right\Vert
_{\infty}\left\vert \sigma(t,X_{t}^{u},v_{t})-\sigma(t,X_{t}^{u}%
,u_{t})\right\vert \\
& +\left\vert f(t,X_{t}^{u},Y_{t}^{u},Z_{t}^{u},v_{t})-f(t,X_{t}^{u},Y_{t}%
^{u},Z_{t}^{u},u_{t})\right\vert ,
\end{array}
\]
by (\ref{est-Xv-Xu}) and using a standard BSDE estimate, we get
\begin{equation}%
\begin{array}
[c]{l}%
\mathbb{\tilde{E}}\left[  \sup\limits_{t\in\lbrack0,T]}\left\vert Y_{t}%
^{v}-Y_{t}^{u}\right\vert ^{2}+\int_{0}^{T}\left\vert Z_{t}^{v}-Z_{t}%
^{u}\right\vert ^{2}dt\right]  \\
\leq C\mathbb{\tilde{E}}\left[  \int_{0}^{T}\left\vert b(t,X_{t}^{u}%
,v_{t})-b(t,X_{t}^{u},u_{t})\right\vert ^{2}dt+\int_{0}^{T}\left\vert
\sigma(t,X_{t}^{u},v_{t})-\sigma(t,X_{t}^{u},u_{t})\right\vert ^{2}dt\right.
\\
\ \ +\left.  \int_{0}^{T}\left\vert f(t,X_{t}^{u},Y_{t}^{u},Z_{t}^{u}%
,v_{t})-f(t,X_{t}^{u},Y_{t}^{u},Z_{t}^{u},u_{t})\right\vert ^{2}dt\right]  .
\end{array}
\label{est-YZv-YZu}%
\end{equation}
Thus, by Lemma \ref{adj-1st-lem}, the estimate (\ref{est-YZv-YZu}) and $\left\Vert D^{2}f\right\Vert _{\infty
}<\infty$, (\ref{est-fzz-Z-delata-square}) can be
dominated by%
\[%
\begin{array}
[c]{l}%
2\left\Vert D^{2}f\right\Vert _{\infty}\mathbb{\tilde{E}}\left[  \int_{0}%
^{T}\left\vert \left(  Z_{t}^{v}\right)  ^{i}-\left(  Z_{t}^{u}\right)
^{i}-\tilde{\Delta}^{i}(t,X_{t}^{u},p_{t}^{u},v_{t},u_{t})\right\vert
^{2}dt\right]  \\
\leq C\mathbb{\tilde{E}}\left[  \int_{0}^{T}\left\vert Z_{t}^{v}-Z_{t}%
^{u}\right\vert ^{2}dt+\int_{0}^{T}\left\vert \tilde{\Delta}(t,X_{t}^{u}%
,p_{t}^{u},v_{t},u_{t})\right\vert ^{2}dt\right]  \\
\leq C\mathbb{\tilde{E}}\left[  \int_{0}^{T}\left\vert b(t,X_{t}^{u}%
,v_{t})-b(t,X_{t}^{u},u_{t})\right\vert ^{2}dt+\int_{0}^{T}\left\vert
\sigma(t,X_{t}^{u},v_{t})-\sigma(t,X_{t}^{u},u_{t})\right\vert ^{2}dt\right.
\\
\ \ +\left.  \int_{0}^{T}\left\vert f(t,X_{t}^{u},Y_{t}^{u},Z_{t}^{u}%
,v_{t})-f(t,X_{t}^{u},Y_{t}^{u},Z_{t}^{u},u_{t})\right\vert ^{2}dt\right]  .
\end{array}
\]

As for (\ref{est-fzz-q-square-X-square}), we first note that
\[%
\begin{array}
[c]{rl}
& \int_{0}^{T}\left\vert \sum\limits_{i,j=1}^{d}f_{z_{i}z_{j}}(t)\left(
X_{t}^{v}-X_{t}^{u}\right)  ^{\intercal}\left(  q_{t}^{u}\right)  ^{i}\left(
\left(  q_{t}^{u}\right)  ^{j}\right)  ^{\intercal}(X_{t}^{v}-X_{t}%
^{u})\right\vert dt\\
\leq & \left\Vert f_{zz}\right\Vert _{\infty}\sum\limits_{i,j=1}^{d}\int%
_{0}^{T}\left\vert \left(  X_{t}^{v}-X_{t}^{u}\right)  ^{\intercal}\left(
q_{t}^{u}\right)  ^{i}\right\vert \cdot\left\vert \left(  \left(  q_{t}%
^{u}\right)  ^{j}\right)  ^{\intercal}(X_{t}^{v}-X_{t}^{u})\right\vert dt\\
\leq & \frac{d}{2}\left\Vert f_{zz}\right\Vert _{\infty}\left(  \sum
\limits_{i=1}^{d}\int_{0}^{T}\left\vert \left(  X_{t}^{v}-X_{t}^{u}\right)
^{\intercal}\left(  q_{t}^{u}\right)  ^{i}\right\vert ^{2}dt+\sum
\limits_{j=1}^{d}\int_{0}^{T}\left\vert \left(  X_{t}^{v}-X_{t}^{u}\right)
^{\intercal}\left(  q_{t}^{u}\right)  ^{j}\right\vert ^{2}dt\right)  \\
\leq & nd\left\Vert f_{zz}\right\Vert _{\infty}\sum\limits_{i=1}^{d}%
\sum\limits_{k=1}^{n}\int_{0}^{T}\left\vert \left(  X_{t}^{v}-X_{t}%
^{u}\right)  ^{k}\left(  q_{t}^{u}\right)  ^{ki}\right\vert ^{2}dt,
\end{array}
\]
where the second inequality comes from $\left \vert ab \right \vert \leq \frac{1}{2} \left( a^{2} + b^{2} \right)$,
and the last inequality is due to $\left\vert \sum_{k=1}^{n}a_{k}\right\vert ^{2}\leq n\sum_{k=1}^{n}\left\vert a_{k}\right\vert ^{2}$.
Then, by Proposition \ref{prop-Hp-BMO} and (\ref{equiv-measure-BMO}),
we have
\[%
\begin{array}
[c]{rl}
& \mathbb{\tilde{E}}\left[  \int_{0}^{T}\left\vert \sum\limits_{i,j=1}%
^{d}f_{z_{i}z_{j}}(t)\left(  X_{t}^{v}-X_{t}^{u}\right)  ^{\intercal}\left(
q_{t}^{u}\right)  ^{i}\left(  \left(  q_{t}^{u}\right)  ^{j}\right)
^{\intercal}(X_{t}^{v}-X_{t}^{u})\right\vert dt\right]  \\
\leq & nd\left\Vert f_{zz}\right\Vert _{\infty}\sum\limits_{i=1}^{d}%
\sum\limits_{k=1}^{n}\mathbb{\tilde{E}}\left[  \left\langle \left(  X_{{}}%
^{v}-X_{{}}^{u}\right)  ^{k}\cdot\left(  \left(  q_{{}}^{u}\right)  ^{ki}\cdot
W^{i}\right)  \right\rangle _{T}\right]  \\
\leq & 2nd\left\Vert f_{zz}\right\Vert _{\infty}\sum\limits_{i=1}^{d}%
\sum\limits_{k=1}^{n}\left\Vert \left(  q_{{}}^{u}\right)  ^{ji}\cdot\tilde
{W}^{i}\right\Vert _{\mathrm{BMO}(\mathbb{\tilde{P}})}^{2}\mathbb{\tilde{E}}\left[
\sup\limits_{t\in\lbrack0,T]}\left\vert \left(  X_{t}^{v}-X_{t}^{u}\right)
^{j}\right\vert ^{2}\right]  \\
\leq & 2c_{2}n^{2}d^{2}\left\Vert f_{zz}\right\Vert _{\infty}\left\Vert
q^{u}\right\Vert _{\kappa}^{2}\mathbb{\tilde{E}}\left[  \sup\limits_{t\in
\lbrack0,T]}\left\vert X_{t}^{v}-X_{t}^{u}\right\vert ^{2}\right]  \\
\leq & C\mathbb{\tilde{E}}\left[  \int_{0}^{T}\left\vert b(t,X_{t}^{u}%
,v_{t})-b(t,X_{t}^{u},u_{t})\right\vert ^{2}dt+\int_{0}^{T}\left\vert
\sigma(t,X_{t}^{u},v_{t})-\sigma(t,X_{t}^{u},u_{t})\right\vert ^{2}dt\right].
\end{array}
\]
The estimates for other terms in $\mathbb{\tilde{E}}\left[  \int_{0}%
^{T}\left\vert R_{4}(t)\right\vert dt\right]  $ are similar to either
(\ref{est-fzz-Z-delata-square}) or (\ref{est-fzz-q-square-X-square}). Hence, we obtain%
\[%
\begin{array}
[c]{ll}%
\mathbb{\tilde{E}}\left[  \int_{0}^{T}\left\vert R_{4}(t)\right\vert
dt\right]  \leq & C\mathbb{\tilde{E}}\left[  \int_{0}^{T}\left\vert
b(t,X_{t}^{u},v_{t})-b(t,X_{t}^{u},u_{t})\right\vert ^{2}dt\right.  \\
& +\int_{0}^{T}\left\vert \sigma(t,X_{t}^{u},v_{t})-\sigma(t,X_{t}^{u}%
,u_{t})\right\vert ^{2}dt\\
& +\left.  \int_{0}^{T}\left\vert f(t,X_{t}^{u},Y_{t}^{u},Z_{t}^{u}%
,v_{t})-f(t,X_{t}^{u},Y_{t}^{u},Z_{t}^{u},u_{t})\right\vert ^{2}dt\right]  .
\end{array}
\]

$R_{6}(\cdot)$ and other remained terms in $\phi$ can be estimated as in the way of dealing with the above three parts.
Consequently, we have%
\begin{equation}%
\begin{array}
[c]{l}%
\mathbb{\tilde{E}}\left[  \left\vert \eta(T)\right\vert +\int_{0}%
^{T}\left\vert \phi_{t}\right\vert dt\right]  \\
\leq C\left\{  \sum\limits_{\psi\in\{b,\sigma\}}\mathbb{\tilde{E}}\left[
\int_{0}^{T}\left\vert \psi(t,X_{t}^{u},v_{t})-\psi^{u}(t)\right\vert
^{2}dt\right]  +\mathbb{\tilde{E}}\left[  \int_{0}^{T}\left\vert
f(t,\Theta_{t}^{u},v_{t})-f^{u}(t)\right\vert ^{2}dt\right]  \right.  \\
\ \ +\left.  \sum\limits_{w\in\{x,y,z\}}\mathbb{\tilde{E}}\left[  \int_{0}%
^{T}\left\vert G_{w}(t,\Theta_{t}^{u},p_{t}^{u},q_{t}^{u},v_{t},u_{t}%
)-G_{w}(t,\Theta_{t}^{u},p_{t}^{u},q_{t}^{u},u_{t},u_{t})\right\vert
^{2}dt\right]  \right\}  .
\end{array}
\label{est-phi}%
\end{equation}
Combining (\ref{est-phi}) with (\ref{Jv-Ju-origin}), we finally obtain (\ref{est-Jv-Ju}), which completes the proof.
\end{proof}

For any integer $m$, recall the returned control $u^{m-1}(\cdot)$ at the $m$th iteration in Algorithm \ref{algorithm-MSA},
the corresponding state trajectory $\Theta^{m} = \left(  X^{m},Y^{m},Z^{m}\right)$, the first and second-order adjoint processes $\left( p^{m},q^{m}  \right)$, $P^{m}$
and other notations defined in (\ref{def-notation}).
Define a new probability $\mathbb{P}^{m}$ by $d\mathbb{P}^{m}:=\Xi_{T}^{m}  d\mathbb{P}$,
where $\Xi_{t}^{m}:=\mathcal{E}\left(  \sum_{i=1}^{d}\int%
_{0}^{t}f_{z_{i}}^{u^{m-1}}(s)dW_{s}^{i}\right) , $
and denote by $\mathbb{E}^{m}\left[  \cdot\right]  $ the mathematical expectation with respect to $\mathbb{P}^{m}$.
Set
\[
\mathcal{\hat{H}}_{m}\mathcal{(}t\mathcal{)}=\mathcal{H}(t,\Theta_{t}%
^{m},p_{t}^{m},q_{t}^{m},P_{t}^{m},u_{t}^{m},u_{t}^{m-1})-\mathcal{H}%
(t,\Theta_{t}^{m},p_{t}^{m},q_{t}^{m},P_{t}^{m},u_{t}^{m-1},u_{t}^{m-1})
\]
and $\mu_{m}=\mathbb{E}^{m}\left[  \int_{0}^{T}\mathcal{\hat{H}}_{m}\mathcal{(}%
t\mathcal{)}dt\right]  $. Now we state the main result of the paper.

\begin{theorem}
\label{thm-convergence} Let Assumption \ref{assum-1} hold.
Then, for $\rho>2C e^{  \left\Vert f_{y} \right\Vert _{\infty} T  } $,
the sequence $\{J(u^{m}(\cdot))\}_{m}$ obtained by Algorithm \ref{algorithm-MSA} converges to a local minimum of (\ref{state-eq})-(\ref{cost-func}) and
$\lim_{m \rightarrow \infty} \mu_{m} = 0$, where $C$ is the constant determined by the estimate (\ref{est-Jv-Ju}).
Moreover, recalling the $\epsilon$ introduced in Algorithm \ref{algorithm-MSA},
then there exists the smallest positive integer $m_{\epsilon}$ depending on $\epsilon$ such that, for all $m\geq m_{\epsilon}$ and all $u(\cdot) \in \mathcal{U}[0,T]$,
\begin{equation}%
\begin{array}
[c]{rl}
& \mathbb{E}^{m}\left[  \int_{0}^{T}\mathcal{\tilde{H}}(t,\Theta_{t}^{m}%
,p_{t}^{m},q_{t}^{m},P_{t}^{m},u_{t},u_{t}^{m-1})\right.  \\
& -\left.  \mathcal{\tilde{H}}(t,\Theta_{t}^{m},p_{t}^{m},q_{t}^{m},P_{t}%
^{m},u_{t}^{m-1},u_{t}^{m-1})dt\right]  \\
\geq & -\left(  e^{-\left\Vert f_{y}\right\Vert _{\infty}T}
-\frac{2C}{\rho}\right)  ^{-1}\epsilon.
\end{array}
\label{delta-H-minimum}%
\end{equation}
\end{theorem}

\begin{proof}[\textbf{Proof}]
At first, from the updating step (\ref{def-u-n-step}) in Algorithm \ref{algorithm-MSA}, we get $\mathcal{\hat{H}}_{m}%
\mathcal{(}t\mathcal{)}\leq0$, $dt \otimes d\mathbb{P}$-$a.e.$,
so $\mu_{m}\leq0$. By Lemma \ref{lem-costfunc-dominate}, letting $v(\cdot)=u^{m}(\cdot)$, $u(\cdot)=u^{m-1}(\cdot)$
and noting that $\mathcal{\hat{H}}_{m}^{+}(\cdot)=0$, then we have%
\begin{equation}%
\begin{array}
[c]{l}%
J\left(  u^{m}(\cdot)\right)  -J\left(  u^{m-1}(\cdot)\right)  \\
\leq e^{-\left\Vert f_{y}\right\Vert _{\infty}T}\mathbb{E}^{m}\left[  \int%
_{0}^{T}\mathcal{\hat{H}}_{m}\mathcal{(}t\mathcal{)}dt\right]  +\mathbb{E}%
^{m}\left[  \int_{0}^{T}\left\vert f(t,\Theta_{t}^{m},u_{t}^{m})-f(t,\Theta
_{t}^{m},u_{t}^{m-1})\right\vert ^{2}dt\right]  \\
\ \ +\sum\limits_{\psi\in\{b,\sigma\}}\mathbb{E}^{m}\left[  \int_{0}%
^{T}\left\vert \psi(t,X_{t}^{m},u_{t}^{m})-\psi(t,X_{t}^{m},u_{t}%
^{m-1})\right\vert ^{2}dt\right]  \\
\ \ +\ C\left\{  \sum\limits_{w\in\{x,y,z\}}\mathbb{E}^{m}\left[  \int_{0}%
^{T}\left\vert G_{w}(t,\Theta_{t}^{m},p_{t}^{m},q_{t}^{m},u_{t}^{m}%
,u_{t}^{m-1})\right.  \right.  \right.  \\
\ \ \ \ \ \ \ \ \ \ \ \ \ \ \ \ \ \ \ \ \ \ \ \ \ \ \ \ \ \ \ \left.  \left.
-\left.  G_{w}(t,\Theta_{t}^{m},p_{t}^{m},q_{t}^{m},u_{t}^{m-1},u_{t}%
^{m-1})\right\vert ^{2}dt\right]  \right\}
\end{array}
\label{nstep-n-1step}%
\end{equation}
for some universal constant $C>0$ depending on $n$, $d$, $T$,
$\left\Vert b_{x}\right\Vert _{\infty}$, $\left\Vert \sigma_{x}\right\Vert
_{\infty}$, $\left\Vert \Phi_{x}\right\Vert _{\infty}$, $\left\Vert b_{xx}\right\Vert _{\infty}$, $\left\Vert \sigma_{xx}\right\Vert
_{\infty}$, $\left\Vert \Phi_{xx}\right\Vert _{\infty}$,
$\left\Vert
Df\right\Vert _{\infty}$ and $\left\Vert D^{2}f\right\Vert _{\infty}$.
Hence, by choosing $\rho>2C e^{\left\Vert f_{y} \right\Vert _{\infty} T}$
and the definition of $\mathcal{\tilde{H}}$, (\ref{nstep-n-1step}) implies
\begin{equation}
\label{Jum-Jum-1-decrease}
J\left(  u^{m}(\cdot)\right)  -J\left(  u^{m-1}(\cdot)\right)  \leq\left(
e^{-\left\Vert f_{y}\right\Vert _{\infty}T} - \frac{2C}{\rho}\right)  \mu_{m}\leq 0 .
\end{equation}
Consequently, for any integer $l \geq 1$, we have%
\[%
\begin{array}
[c]{rl}%
\sum\limits_{m=1}^{l}\left(  -\mu_{m}\right)  \leq & \left(  e^{-\left\Vert f_{y}\right\Vert _{\infty}T}  -\dfrac{2C}{\rho}\right)
^{-1}\sum\limits_{m=1}^{l}\left[  J\left(  u^{m-1}(\cdot)\right)  -J\left(
u^{m}(\cdot)\right)  \right]  \\
= & \left(  e^{-\left\Vert f_{y}\right\Vert _{\infty}T}
-\dfrac{2C}{\rho}\right)  ^{-1}\left[  J\left(  u^{0}(\cdot)\right)  -J\left(
u^{l}(\cdot)\right)  \right]  \\
\leq & \left(  e^{-\left\Vert f_{y}\right\Vert _{\infty}T}
-\dfrac{2C}{\rho}\right)  ^{-1}\left[  J\left(  u^{0}(\cdot)\right)
-\inf\limits_{u(\cdot)\in\mathcal{U}[0,T]}J\left(  u(\cdot)\right)  \right]
\\
< & \infty,
\end{array}
\]
which implies that $\sum\limits_{m=1}^{\infty}\left(  -\mu_{m}\right)  < \infty$.
Since $-\mu_{m}\geq0$, we obtain that $\mu_{m}\rightarrow0$ as $m\rightarrow \infty$.

As for the last claim (\ref{delta-H-minimum}), since $\inf_{u(\cdot)\in\mathcal{U}[0,T]}J\left(  u(\cdot)\right) > -\infty$, for any sufficiently small $\epsilon >0$,
$J(u^{m-1}(\cdot))-J(u^{m}(\cdot)) < \epsilon$ as long as $m$ is large enough.
Denote by $m_{\epsilon}$ the smallest positive integer such that, for all $m\geq m_{\epsilon}$, $J(u^{m-1}(\cdot))-J(u^{m}(\cdot)) < \epsilon$.
Then it follows from (\ref{def-general-Hamiltonian}) and (\ref{Jum-Jum-1-decrease}) that, for all $m\geq m_{\epsilon}$ all $u(\cdot) \in \mathcal{U}[0,T]$,
\[%
\begin{array}
[c]{rl}
& \mathbb{E}^{m}\left[  \int_{0}^{T}\mathcal{\tilde{H}}(t,\Theta_{t}^{m}%
,p_{t}^{m},q_{t}^{m},P_{t}^{m},u_{t},u_{t}^{m-1})\right.  \\
& -\left.  \mathcal{\tilde{H}}(t,\Theta_{t}^{m},p_{t}^{m},q_{t}^{m},P_{t}%
^{m},u_{t}^{m-1},u_{t}^{m-1})dt\right]  \\
\geq & \mathbb{E}^{m}\left[  \int_{0}^{T}\mathcal{\tilde{H}}(t,\Theta_{t}%
^{m},p_{t}^{m},q_{t}^{m},P_{t}^{m},u_{t}^{m},u_{t}^{m-1})\right.  \\
& -\left.  \mathcal{\tilde{H}}(t,\Theta_{t}^{m},p_{t}^{m},q_{t}^{m},P_{t}%
^{m},u_{t}^{m-1},u_{t}^{m-1})dt\right]  \\
\geq & \mu_{m}\\
\geq & -\left(  e^{-\left\Vert f_{y}\right\Vert _{\infty}T}
-\frac{2C}{\rho}\right)  ^{-1}\epsilon,
\end{array}
\]
which completes the proof.
\end{proof}

\begin{remark}
In Theorem \ref{thm-suffi-cond} below, we will prove that (\ref{delta-H-minimum}) is sufficient for making
$u^{m_{\epsilon}-1}(\cdot)$ nearly minimize $J$ with an order of $\epsilon^{\frac{1}{2}}$ in a special case.
\end{remark}

\begin{corollary}
\label{cor-1} Suppose Assumption \ref{assum-1} holds, and $f$ is
independent of $y,z$. We further assume that the following local maximum
principle holds:%
\begin{equation}
G(t,\bar{X}_{t},\bar{p}_{t},\bar{q}_{t},\bar{u}_{t})\leq G(t,\bar{X}_{t}%
,\bar{p}_{t},\bar{q}_{t},u), \ \text{\ \ }\forall u\in U,\text{ }%
dt \otimes d\mathbb{P}\text{-a.e.}. \label{local-SMP}%
\end{equation}
Then the estimate (\ref{est-Jv-Ju})\ becomes%
\begin{equation}%
\begin{array}
[c]{l}%
J\left(  v(\cdot)\right)  -J\left(  u(\cdot)\right)  \\
\leq\mathbb{E}\left[  \int_{0}^{T}\left[  G(t,X_{t}^{u},p_{t}^{u},q_{t}%
^{u},v_{t})-G(t,X_{t}^{u},p_{t}^{u},q_{t}^{u},u_{t})\right]  dt\right]  \\
\ \ +C\left\{  \sum\limits_{\psi\in\{b,\sigma\}}\mathbb{E}\left[  \int_{0}%
^{T}\left\vert \psi(t,X_{t}^{u},v_{t})-\psi(t,X_{t}^{u},u_{t})\right\vert
^{2}dt\right]  \right.  \\
\ \ +\left.  \mathbb{E}\left[  \int_{0}^{T}\left\vert G_{x}(t,X_{t}^{u}%
,p_{t}^{u},q_{t}^{u},v_{t})-G_{x}(t,X_{t}^{u},p_{t}^{u},q_{t}^{u}%
,u_{t})\right\vert ^{2}dt\right]  \right\}  .
\end{array}
\label{est-f-independent-yz}%
\end{equation}
Furthermore, the sequence $\{J(u^{m}(\cdot))\}_{m}$ obtained by Algorithm \ref{algorithm-MSA} converges to a local minimum of (\ref{state-eq})-(\ref{cost-func}) for $\rho>2C$.
\end{corollary}

\begin{proof}[\textbf{Proof}]
Since $f$ is independent of $y,z$ and $P$ vanishes, we have $\mathbb{\tilde{E}[\cdot]=E[\cdot]}$
and $\mathcal{H}(t,x,y,z,p,q,P,v,u)=G(t,x,p,q,v)$. In such case, the augmented Hamiltonian becomes
\[%
\begin{array}
[c]{cl}%
\mathcal{\tilde{H}(}t,x,p,q,v,u\mathcal{)}= & G(t,x,p,q,v)\\
& +\dfrac{\rho}{2}\left\vert b(t,x,v)-b(t,x,u)\right\vert ^{2}+\dfrac{\rho}%
{2}\left\vert \sigma(t,x,v)-\sigma(t,x,u)\right\vert ^{2}\\
& +\dfrac{\rho}{2}\left\vert G_{x}(t,x,p,q,v)-G_{x}(t,x,p,q,u)\right\vert
^{2}.
\end{array}
\]
Then, similar to the proof of Lemma \ref{lem-costfunc-dominate}, we update the
control by%
\[
u_{t}^{m}\in\arg\min_{u\in U}\mathcal{\tilde{H}(}t,X_{t}^{m},p_{t}^{m}%
,q_{t}^{m},u,u_{t}^{m-1}\mathcal{)}.
\]
and obtain estimate (\ref{est-f-independent-yz}).
By Theorem \ref{thm-convergence}, we get the convergence of Algorithm
\ref{algorithm-MSA} for $\rho>2C$.
\end{proof}

\begin{remark}
Without assuming that $D_{x}^{2}\sigma(t,x,u)=0$ for all $(t,x,u)\in
\lbrack0,T]\times\mathbb{R}^{n}\times U$, estimate (\ref{est-f-independent-yz}%
) is same as the one obtained by Lemma 2.3 in \cite{BDL-MSA-2020}. Actually,
control system (\ref{state-eq}) degenerates into the one studied in
\cite{BDL-MSA-2020} when $f$ is independent of $y,z$. Hence, we relax their
assumption and our model is a general case of theirs.
\end{remark}

\subsection{Convergence Rate in A Special Case}
In this section, we provide a case where the convergence rate is available.
Consider the following stochastic control problem:
over $\mathcal{U}[0,T]$, minimize
\begin{equation}
\label{eg-quasi-LQ-cost-func}
J\left(  u(\cdot)\right)  :=\frac{1}{2}\mathbb{E}\left[  \left(  X_{T}%
^{u}\right)  ^{\intercal}\Gamma X_{T}^{u}+\int_{0}^{T}\left[  \left(
X_{t}^{u}\right)  ^{\intercal}A_{t}X_{t}^{u}+\left(  u_{t}\right)
^{\intercal}B_{t}u_{t}\right]  dt\right]
\end{equation}
subject to
\begin{equation}
\left\{
\begin{array}
[c]{rl}%
dX_{t}^{u}= & \left(  b_{1}(t)X_{t}^{u}+b_{2}(t)\right)  dt+\sigma
(t,u_{t})dW_{t},\quad t\in\lbrack0,T],\\
X_{0}^{u}= & x_{0},
\end{array}
\right.  \label{eg-quasi-LQ-eq}%
\end{equation}
where $x_{0}\in\mathbb{R}^{n}$, $\Gamma\in\mathbb{S}^{n\times n}$,
$\sigma:[0,T]\times U\longmapsto\mathbb{R}^{n\times d}$; $b_{1}$ is an $n\times n$
matrix-valued, bounded, deterministic process; $b_{2}$ is an $n$-dimensional,
vector-valued, bounded, deterministic process; $A$, $B$ are respectively $n\times n$,
$k\times k$ matrix-valued, symmetric, bounded, deterministic processes.

Observe that, for any $u(\cdot)\in \mathcal{U}[0,T]$, (\ref{1st-adj-eq}) becomes
\[
p_{t}^{u}=\Gamma X_{T}^{u}+\int_{t}^{T}\left(  b_{1}^{\intercal}(s)p_{s}%
^{u}+A_{s}X_{s}^{u}\right)  ds-\sum_{i=1}^{d}\int_{t}^{T}\left(  q_{s}%
^{u}\right)  ^{i}dW_{s}^{i},\text{ \ }t\in\lbrack0,T],
\]%
and (\ref{2nd-adj-eq}) becomes
\[
P_{t}=\Gamma+\int_{t}^{T}\left[  b_{1}^{\intercal}(s)P_{s}+\left(  P_{s}\right)  ^{\intercal}b_{1}^{\intercal}(s)+A_{s}\right]  ds,\text{ \ }t\in\lbrack0,T].
\]
The Hamiltonian $\mathcal{H}$ turns out to be
\[%
\begin{array}
[c]{rl}%
\mathcal{H}(t,x,p,q,P,v,u)= & p^{\intercal}[b_{1}(t)x+b_{2}(t)]+\sum\limits_{i=1}^{d}\left(
q^{i}\right)  ^{\intercal}\sigma(t,v)+\frac{1}{2}\left(  x^{\intercal}%
A_{t}x+v^{\intercal}B_{t}v\right)  \\
& +\frac{1}{2}\left(  \sigma(t,v)-\sigma(t,u)\right)  ^{\intercal}P\left(
\sigma(t,v)-\sigma(t,u)\right)  .
\end{array}
\]
One can verify that, in (\ref{eta-1}), $\eta(T)=R_{1}(s)=R_{3}(s)=R_{4}(s)=R_{5}(s)=R_{6}(s)=0$, and
\[
G_{x}(s,X_{s}^{u},p_{s}^{u},q_{s}^{u},v_{s},u_{s})-G_{x}(s,X_{s}^{u},p_{s}%
^{u},q_{s}^{u},u_{s},u_{s})=0,
\]
which implies $\phi_{s}=0$ in (\ref{eta-2}).
Thus, (\ref{eta-2}) becomes
\[
\eta(t)=\int_{t}^{T}\mathcal{\hat{H}(}s\mathcal{)}ds-\sum
\limits_{i=1}^{d}\int_{t}^{T}\zeta^{i}(s)dW_{s}^{i},\text{ \ }t\in\lbrack0,T].
\]
Particularly, we have
\begin{equation}
\label{LQ-Jv-Ju}
J\left(  v(\cdot)\right)  -J\left(  u(\cdot)\right)  =\eta(0)=\mathbb{E}%
\left[  \int_{0}^{T}\mathcal{\hat{H}}(t)dt\right]  ,
\end{equation}
where $\mathcal{\hat{H}}(t)=\mathcal{H(}t,X_{t}^{u},p_{t}^{u},q_{t}^{u}%
,P_{t},v_{t},u_{t}\mathcal{)-H(}t,X_{t}^{u},p_{t}^{u},q_{t}^{u},P_{t}%
,u_{t},u_{t}\mathcal{)}$, $t\in\lbrack0,T]$.

Compared with (\ref{est-Jv-Ju}), the universal constant $C$ disappears in (\ref{LQ-Jv-Ju}).
Consequently, in Algorithm \ref{algorithm-MSA}, only if we update the control by $\mathcal{H}$ instead of $\tilde{\mathcal{H}}$ (i.e. $\rho=0$ in such case),
we can make $J$ decrease efficiently. Furthermore, we can obtain a $\frac{1}{m}$-order convergence rate if (\ref{eg-quasi-LQ-cost-func})-(\ref{eg-quasi-LQ-eq}) admits an optimal control $\bar{u}(\cdot)$.
This is illustrated by the following theorem.
\begin{theorem}
\label{thm-LQ-convergence-rate}
Let $\sigma$ satisfy (i) in Assumption \ref{assum-1}. Assume that (\ref{eg-quasi-LQ-cost-func})-(\ref{eg-quasi-LQ-eq}) admits an optimal control $\bar{u}(\cdot)\in \mathcal{U}[0,T]$.
Then $\tilde{\mathcal{H}} = \mathcal{H}$ so (\ref{def-u-n-step}) in Algorithm \ref{algorithm-MSA} becomes
\begin{equation}
\label{LQ-update-control-step}
u_{t}^{m} \in \arg\min_{u\in U}\mathcal{H}(t,X_{t}^{m},p_{t}^{m},q_{t}^{m},%
P_{t},u,u_{t}^{m-1}), \quad t\in\lbrack0,T],
\end{equation}
and there exists an positive integer $m_{0}$ such that,
for all $m \geq m_{0}$,
\begin{equation}
\label{rate-1-m}
0\leq J(u^{m-1}(\cdot))-J(\bar{u}(\cdot)) \leq \frac{C_{1}}{m},
\end{equation}
where $C_{1}=\max\left \{ J(u^{m_{0}-1}(\cdot))-J(\bar{u}(\cdot)), 1 \right \}$.
\end{theorem}

\begin{proof}
Similar to the proof of Theorem \ref{thm-convergence}, due to (\ref{LQ-update-control-step}), we have
$0 \geq \mu_{m}=\mathbb{E}\left[ \int_{0}^{T} \mathcal{\hat{H}}_{m}(t)dt \right] \rightarrow 0$ as $m \rightarrow \infty$,
where $\mathcal{\hat{H}}_{m}(t)=\mathcal{H}(t,X_{t}^{m},p_{t}^{m},q_{t}^{m}%
,P_{t},u_{t}^{m},u_{t}^{m-1})$ $-\mathcal{H}(t,X_{t}^{m},p_{t}^{m},q_{t}^{m}
,P_{t},u_{t}^{m-1},u_{t}^{m-1})$.
Thus, there exists a positive integer $m_{0}$ such that, for all $m \geq m_{0}$,
$0\leq -\mu_{m} < \frac{1}{2}$. By Lemma 3.2 in \cite{BDL-MSA-2020}, for any $m \geq m_{0}$,
there exists $t_{m} \in [0,T]$ such that
\begin{equation}
\label{small-interval}
\mathbb{E}\left[ \int_{I_{t_{m},\mu_{m}}} \mathcal{\hat{H}}_{m}(t)dt \right] \leq -\mu_{m}^{2},
\end{equation}
where $I_{t_{m},\mu_{m}}=\left[ t_{m}-\left \vert \mu_{m} \right \vert T, t_{m}+\left \vert \mu_{m} \right \vert T \right] \bigcap [0,T]$.

From (\ref{LQ-Jv-Ju}), (\ref{LQ-update-control-step}) and (\ref{small-interval}), for any $m \geq m_{0}$, we have
\begin{equation}
\label{Jm-difference-square-order}
J\left(u^{m}(\cdot)\right)  -J\left(u^{m-1}(\cdot)\right)  =\mathbb{E}%
\left[  \int_{0}^{T}\mathcal{\hat{H}}_{m}(t)dt\right]  \leq\mathbb{E}\left[
\int_{I_{t_{m},\mu_{m}}}\mathcal{\hat{H}}_{m}(t)dt\right]  \leq -\mu_{m}^{2}.
\end{equation}
On the other hand, since we assume $\bar{u}(\cdot)$ is optimal, applying (\ref{LQ-Jv-Ju}) to $u^{m-1}(\cdot)$, $\bar{u}(\cdot)$,
we obtain
\begin{equation}%
\begin{array}
[c]{l}%
\ \ \ J\left(  u^{m-1}(\cdot)\right)  -J\left(  \bar{u}(\cdot)\right)  \\
=-\left[  J\left(  \bar{u}(\cdot)\right)  -J\left(  u^{m-1}(\cdot)\right)
\right]  \\
=\mathbb{E}\left[  \int_{0}^{T}\left[  \mathcal{H(}t,X_{t}^{m},p_{t}^{m}%
,q_{t}^{m},P_{t},u_{t}^{m-1},u_{t}^{m-1})-\mathcal{H(}t,X_{t}^{m},p_{t}%
^{m},q_{t}^{m},P_{t},\bar{u}_{t},u_{t}^{m-1}\mathcal{)}\right]  dt\right]  \\
\leq\mathbb{E}\left[  \int_{0}^{T}\left[  \mathcal{H(}t,X_{t}^{m},p_{t}%
^{m},q_{t}^{m},P_{t},u_{t}^{m-1},u_{t}^{m-1}\mathcal{)}-\mathcal{H(}%
t,X_{t}^{m},p_{t}^{m},q_{t}^{m},P_{t},u_{t}^{m},u_{t}^{m-1}\mathcal{)}\right]
dt\right]  \\
=-\mu_{m}.
\end{array}
\label{Jum-Jbaru-difference}%
\end{equation}
Set $a_{m}=J\left(  u^{m}(\cdot)\right)  -J\left(  \bar{u}(\cdot)\right)$. Consequently, when $m \geq m_{0}$,
plugging (\ref{Jum-Jbaru-difference}) into (\ref{Jm-difference-square-order}), we have
$a_{m}-a_{m-1}\leq -a_{m-1}^{2}$.
Then, applying Lemma A.1 in \cite{BDL-MSA-2020} to the non-negative series $\left\{ a_{m} \right\}$,
we get $a_{m} \leq \frac{C_{1}}{m}$ with $C_{1}=\max\left \{ J(u^{m_{0}-1}(\cdot))-J(\bar{u}(\cdot)), 1 \right \}$.
\end{proof}

\begin{remark}
One can find that the terminal and running costs in (\ref{eg-quasi-LQ-cost-func}) have a quadratic form in $x$ and $u$, which are beyond the setting of the linear growth in Assumption \ref{assum-1}. Thus the first-order adjoint processes $\left( p^{u}, q^{u} \right)$ are not necessarily bounded in $L_{\mathcal{F}}^{\infty}([0,T];\mathbb{R}^{n}) \times \mathcal{K}(\mathbb{R}^{n\times d})$ across all admissible controls. But we can obtain (\ref{LQ-Jv-Ju}) without using the uniformly bounded property of $\left( p^{u}, q^{u} \right)$.
\end{remark}

\subsection{Finding the Near-Optimal Control in A Special Case}
We give a sufficient condition of the near-optimality for a class of linear forward-backward stochastic recursive problems,
by which we can determine whether the returned control $u^{m_{\epsilon}-1}(\cdot)$ introduced in Theorem \ref{thm-convergence},
for some positive integer $m_{\epsilon}$ depending on $\epsilon$, is a near-optimal control.

Let $U\subset\mathbb{R}^{k}$ is convex and compact. Consider the following
stochastic recursive control problem: over $\mathcal{U}[0,T]$, minimize
$J(u(\cdot)):=Y_{0}^{u}$ subject to
\begin{equation}
\left\{
\begin{array}
[c]{rl}%
dX_{t}^{u}= & \left[  b_{1}(t)X_{t}^{u}+b_{2}(t)u_{t}+b_{3}(t)\right]
dt+\sum\limits_{i=1}^{d}\left[  \sigma_{1}^{i}(t)X_{t}^{u}+\sigma_{2}%
^{i}(t)u_{t}+\sigma_{3}^{i}(t)\right]  dW_{t}^{i},\\
dY_{t}^{u}= & -\left[  f_{1}(t)X_{t}^{u}+f_{2}(t)Y_{t}^{u}+f_{3}%
(t,u_{t})\right]  dt+\left(  Z_{t}^{u}\right)  ^{\intercal}dW_{t},\\
X_{0}^{u}= & x_{0},\text{ \ }Y_{T}^{u}=\alpha^{\intercal} X_{T}^{u}+\gamma,
\end{array}
\right.  \label{example-suffi-cond}%
\end{equation}
where $x_{0},\alpha\in\mathbb{R}^{n}$, $\gamma\in\mathbb{R}$, $b_{1}$, $b_{2}$, $b_{3}$,
$\left(  \sigma_{1}^{i}\right)  _{i=1,\ldots d}$, $\left(  \sigma_{2}%
^{i}\right)  _{i=1,\ldots d}$, $\left(  \sigma_{3}^{i}\right)  _{i=1,\ldots d}$, $f_{1}$, $f_{2}$, $f_{3}$ are all deterministic
processes in suitable sizes. We further assume that

\begin{assumption}
\label{assum-2} (i) $b_{1}$, $b_{2}$, $\left( \sigma_{1}^{i}, \sigma_{2}^{i}, \sigma_{3}^{i} \right)$ for $i=1,\ldots d$,
$f_{1}$, $f_{2}$ are bounded; $f_{3}$ is bounded by $L(1+\left\vert u\right\vert )$
with a given positive number $L$;

(ii) $f_{3}$ is continuously differentiable, convex with respect to $u$;

(iii) $(f_{3})_{u}$ is bounded and Lipschitz continuous with respect to $u$.
\end{assumption}
If Assumption \ref{assum-2} holds, then Assumption \ref{assum-1} holds for
(\ref{example-suffi-cond}) naturally. Thus, under Assumption \ref{assum-2}, by
Theorem \ref{thm-convergence}, the $\epsilon$-minimum condition (\ref{delta-H-minimum}) holds.

Let $u^{m_{\epsilon}-1}(\cdot)$ be the returned control defined in Theorem \ref{thm-convergence}
corresponding to a given permissible error $\epsilon>0$, and $\left(
X^{m_{\epsilon}},Y^{m_{\epsilon}}\right)  $ be the corresponding state trajectory.
The following theorem implies that, under Assumption \ref{assum-2}, (\ref{delta-H-minimum}) is sufficient for the near-optimality of order $\epsilon^{\frac{1}{2}}$.

\begin{theorem}
\label{thm-suffi-cond} Let Assumption \ref{assum-2} hold. Then
\[
J\left(  u_{{}}^{m_{\epsilon}-1}(\cdot)\right)  -\inf_{u(\cdot)\in
\mathcal{U}[0,T]}J\left(  u(\cdot)\right)  \leq\tilde{C}\epsilon^{\frac{1}{2}%
},
\]
where $\tilde{C}$ is a positive constant independent of $\epsilon$.
\end{theorem}

At first, for any $u(\cdot)\in\mathcal{U}[0,T]$, the first-order adjoint
equation (\ref{1st-adj-eq}) degenerates into the following linear ordinary differential equation
(ODE for short)
\[
p_{t}=\alpha+\int_{t}^{T}\left\{  \left[  f_{2}(s)+b_{1}(s)\right]
p_{s}+f_{1}(s)\right\}  ds,\text{ \ }t\in\lbrack0,T],
\]
which admits a unique solution $p\in C^{1}([0,T];\mathbb{R}^{n})$, and the
second-order adjoint equation (\ref{2nd-adj-eq}) vanishes. Then, due to the vanishing of
$z$, $q$ and $P$, we have $\tilde{\mathbb{E}}[\cdot]=\mathbb{E}[\cdot]$ and
$\mathcal{H(}t,x,y,z,p,q,P,v,u\mathcal{)}=G(t,x,y,p,v)$, where $\mathcal{H}$
is introduced by (\ref{def-general-Hamiltonian}), $\tilde{\mathbb{E}}[\cdot]$ is defiend in Lemma \ref{lem-costfunc-dominate}
and
\begin{equation}
\label{suffi-G}
G(t,x,y,p,v)=p^{\intercal}\left[  b_{1}(t)x+b_{2}(t)v + b_{3}(t)\right]  +f_{1}%
(t)x+f_{2}(t)y+f_{3}(t,v).
\end{equation}
For any $u(\cdot)$, $v(\cdot)\in\mathcal{U}[0,T]$, by following the proof of
Lemma \ref{lem-costfunc-dominate}, one can deduce that, in the BSDE (\ref{eta-3}), $\eta(T)=0$, $R_{i}(s)=0$
for $i=1,\ldots,6$ and
\[
G_{x}(s,X_{s}^{u},Y_{s}^{u},p_{s},v_{s})-G_{x}(s,X_{s}^{u},Y_{s}^{u}%
,p_{s},u_{s})=0,
\]
which implies $\phi_{s}=0$ in (\ref{eta-3}). Thus (\ref{eta-3}) becomes
\[
\eta(t)=\int_{t}^{T}\hat{G}(s)ds-\sum\limits_{i=1}^{d}\int_{t}^{T}\zeta
^{i}(s)dW_{s}^{i},\text{ \ }t\in\lbrack0,T],
\]
where $\hat{G}(t)=G(t,X_{t}^{u},Y_{t}^{u},p_{t},v_{t})-G(t,X_{t}^{u},Y_{t}%
^{u},p_{t},u_{t}\mathcal{)}$, $t\in\lbrack0,T]$. Then, similar to (\ref{Jv-Ju-origin}), one
can deduce
\begin{equation}
J\left(  v(\cdot)\right)  -J\left(  u(\cdot)\right)  =\eta(0)\leq
e^{\left\Vert f_{2}\right\Vert_{\infty} T}\mathbb{E}\left[  \int_{0}^{T}\hat{G}%
^{+}(t)dt\right]  -e^{-\left\Vert f_{2}\right\Vert_{\infty} T}\mathbb{E}\left[
\int_{0}^{T}\hat{G}^{-}(t)dt\right]  ,\label{LQ-Jv-Ju-1}%
\end{equation}
where $\hat{G}^{+}(\cdot)$, $\hat{G}^{-}(\cdot)$ are respectively the
positive, negative part of $\hat{G}(\cdot)$. Compared with (\ref{est-Jv-Ju}), the
universal constant $C$ disappears in (\ref{LQ-Jv-Ju-1}). Consequently, in
Algorithm \ref{algorithm-MSA}, only if we update the control by
\begin{equation}
u_{t}^{m}\in\arg\min_{u\in U}G(t,X_{t}^{m},Y_{t}^{m},p_{t}%
,u)\label{control-update-G}%
\end{equation}
instead of minimizing $\tilde{\mathcal{H}}$ (i.e. $\rho=0$ in such case), we
can make $J$ decrease efficiently. Hence, the $\epsilon$-minimum condition
(\ref{delta-H-minimum}) implies
\begin{equation}
\mathbb{E}\left[  \int_{0}^{T}\left[  G(t,X_{t}^{m},Y_{t}^{m},p_{t}%
,u_{t}^{m})-G(t,X_{t}^{m},Y_{t}^{m},p_{t},u_{t}^{m-1})\right]  dt\right]
\geq-e^{\left\Vert f_{2}\right\Vert _{\infty}T}\epsilon.
\label{delta-G-minimum}%
\end{equation}

Introduce the following ODE:
\[
\left\{
\begin{array}
[c]{rl}%
d\Gamma_{t}= & f_{2}(t)\Gamma_{t}dt,\text{ \ }t\in\lbrack0,T],\\
\Gamma_{0}= & 1.
\end{array}
\right.
\]
Since $f_{2}$ is bounded, we have $e^{-\left\Vert f_{2}\right\Vert _{\infty}%
T}\leq\left\Vert \Gamma ^{-1}\right\Vert _{\infty}\leq
e^{\left\Vert f_{2}\right\Vert _{\infty}T}$. Then it follows from
(\ref{control-update-G}) that
\begin{equation}%
\begin{array}
[c]{rl}
& G(t,X_{t}^{m},Y_{t}^{m},p_{t},u_{t}^{m})-G(t,X_{t}^{m},Y_{t}^{m},p_{t},u_{t}%
^{m-1})\\
\leq & e^{-\left\Vert f_{2}\right\Vert _{\infty}T}\Gamma_{t}\left[
G(t,X_{t}^{m},Y_{t}^{m},p_{t},u_{t}^{m})-G(t,X_{t}^{m},Y_{t}^{m},p_{t},u_{t}%
^{m-1})\right]  .
\end{array}
\label{Gamma-difference-H}%
\end{equation}
Therefore, combining (\ref{delta-G-minimum}) with (\ref{Gamma-difference-H}), we
obtain, for all $u(\cdot) \in \mathcal{U}[0,T]$,
\begin{equation}
\mathbb{E}\left[  \int_{0}^{T}\Gamma_{t}\left[  G\mathcal{(}t,X_{t}^{m}%
,Y_{t}^{m},p_{t},u_{t}\mathcal{)}dt-G\mathcal{(}t,X_{t}^{m},Y_{t}^{m}%
,p_{t},u_{t}^{m}\mathcal{)}\right]  dt\right]  \geq-e^{2\left\Vert f_{2}\right\Vert _{\infty}T}\epsilon
.\label{inf-difference-G}%
\end{equation}
One can check that, by applying the formula of integration by parts to $\Gamma p$, the triple $\left( -\Gamma_{t}, \Gamma_{t} p_{t}, 0 \right)_{t\in[0,T]}$ are adjoint processes uniquely solving (9) in \cite{Wang-Guangchen} and then (\ref{inf-difference-G}) is nothing but (16) in Theorem 4.1 in \cite{Wang-Guangchen}. Moreover, since $f_{3}$ is convex in $u$, $G(t,\cdot,\cdot,p,\cdot)$ is also convex in $(x,y,v)$ for any $(t,p) \in [0,T] \times \mathbb{R}^{n}$, which verifies the condition of Theorem 4.1 in \cite{Wang-Guangchen}.
Consequently, applying that theorem, we obtain the desired result.

Particularly, when $x_{0}=\alpha=\gamma=b_{1}=b_{2}=b_{3}=\sigma_{1}^{i}=\sigma_{2}^{i}=\sigma_{3}^{i}=\Phi=f_{1}=0$, $i=1,\ldots,d$, $f_{2}=\beta$ for some given constant $\beta>0$, and $f_{3}$ is independent of $t$, this example is closely related to the standard, continuous-time additive utility model in the theory of the stochastic differential recursive utility (see \cite{Epstein92, Karoui97} for more details).

\section{Numerical Demonstration}
In this section, a numerical demonstration is given to illustrate the general results obtained in the above sections.
\begin{example}
\label{example-1}
Let $n=d=k=1$, $U=\left\{ 0,1 \right\}$, $x_{0}=0$, $T=1$, $b=0$, $\sigma=u$, $f=\sin(Lz)$ and $\Phi=Lx$ for some given constant $0<L\leq \sqrt{\pi}$.
By Theorem \ref{thm-maximum principle}, the corresponding SMP reads
\begin{equation}
\label{example-global-SMP}
\sin \left( L\bar{Z}_{t}+L^{2}(u-\bar{u}_{t}) \right) \geq \sin \left( L\bar{Z}_{t}\right),\quad \forall u \in U, \quad dt \otimes d\mathbb{P}\text{-a.e.},
\end{equation}
where (\ref{1st-adj-eq}) degenerates into a constant equation such that $p_{t}\equiv L$ for all $t \in [0,T]$, and (\ref{2nd-adj-eq}) vanishes.
One can verify (\ref{example-global-SMP}) is also sufficient for the optimality due to the comparison theorem of BSDEs.
Thus it follows from $\sin(L^{2}u) \geq 0, \forall u \in U$ that $\left( \bar{X}, \bar{Y}, \bar{Z}, \bar{u}(\cdot) \right) = \left( 0, 0, 0, 0 \right)$ is an optimal quadruple
and the optimal cost $J(\bar{u}(\cdot))=0$.

The Hamiltonian $\mathcal{H}:[0,1]\times\mathbb{R}\times U \times U \longmapsto \mathbb{R}$ is defined by $\mathcal{H}(t,z,v,u)=\sin \left( Lz+L^{2}(v-u) \right)$.
Then, following from the proof of Lemma \ref{lem-costfunc-dominate}, one can define the augmented Hamiltonian $\tilde{\mathcal{H}}:[0,1]\times\mathbb{R}\times
U\times U\longmapsto\mathbb{R}$ by $\tilde{\mathcal{H}}(t,z,v,u)= \mathcal{H}(t,z,v,u) +\frac{\rho}{2}\left(  v-u\right)  ^{2}$, where
$\rho=10L^{4}\left[ 1+ (1+L^{2})(1+8L^{2}e^{8L^{2}}) \right]$ is determined by a careful estimate of the constant $C$ appearing in (\ref{est-Jv-Ju}) of Lemma \ref{lem-costfunc-dominate}.
\end{example}

In our numerical computation, we discretize the time interval $[0,1]$ into $20$ intervals $\Delta_{i}:=[\frac{i-1}{20}, \frac{i}{20}]$, $i=1,\ldots,20$.
By generating random numbers with values $0$ or $1$ on each $\Delta_{i}$,
we can approximately get an initial control $u^{0}(\cdot)$ over $[0,1]$. Then we put this $u^{0}(\cdot)$ into the program based on the Monte Carlo algorithm to solve the FBSDE in (\ref{state-eq})-(\ref{cost-func}) numerically (please refer to \cite{Gobet} for the details).
The following two figures illustrate the performance of Algorithm \ref{algorithm-MSA} corresponding to different choices of $L$.
The horizontal and vertical coordinates represent the times of iterations $m$ and the values of the cost $J(u^{m-1}(\cdot))$.

\begin{figure}[h]
  \centering
  \begin{minipage}{0.48 \linewidth}
  \vspace{3pt}
  \centerline{\includegraphics[height=5cm,width=6.5cm]{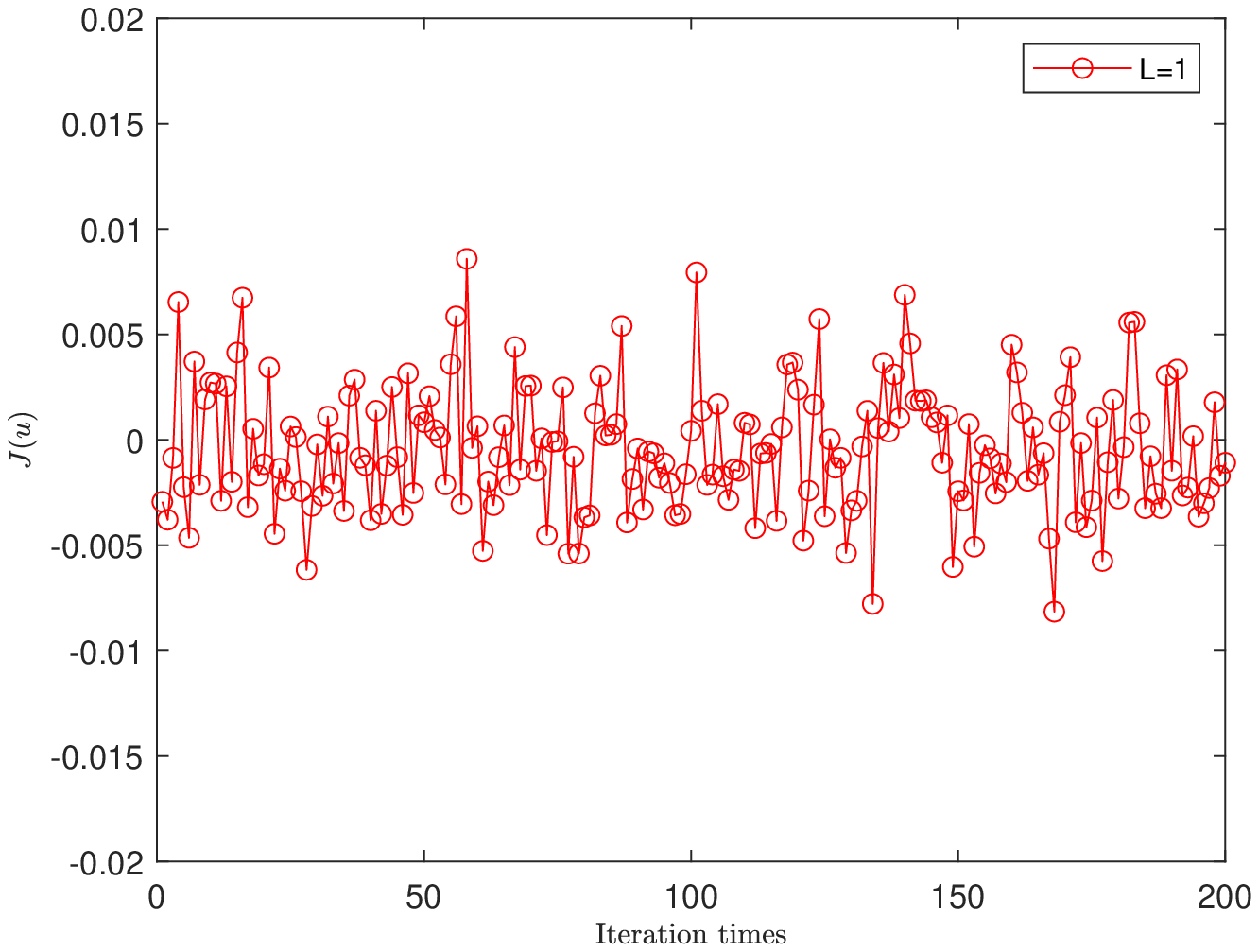}}
  \centerline{The case $L=1$.}
  \end{minipage}
  \begin{minipage}{0.48 \linewidth}
  \vspace{3pt}
  \centerline{\includegraphics[height=5cm,width=6.5cm]{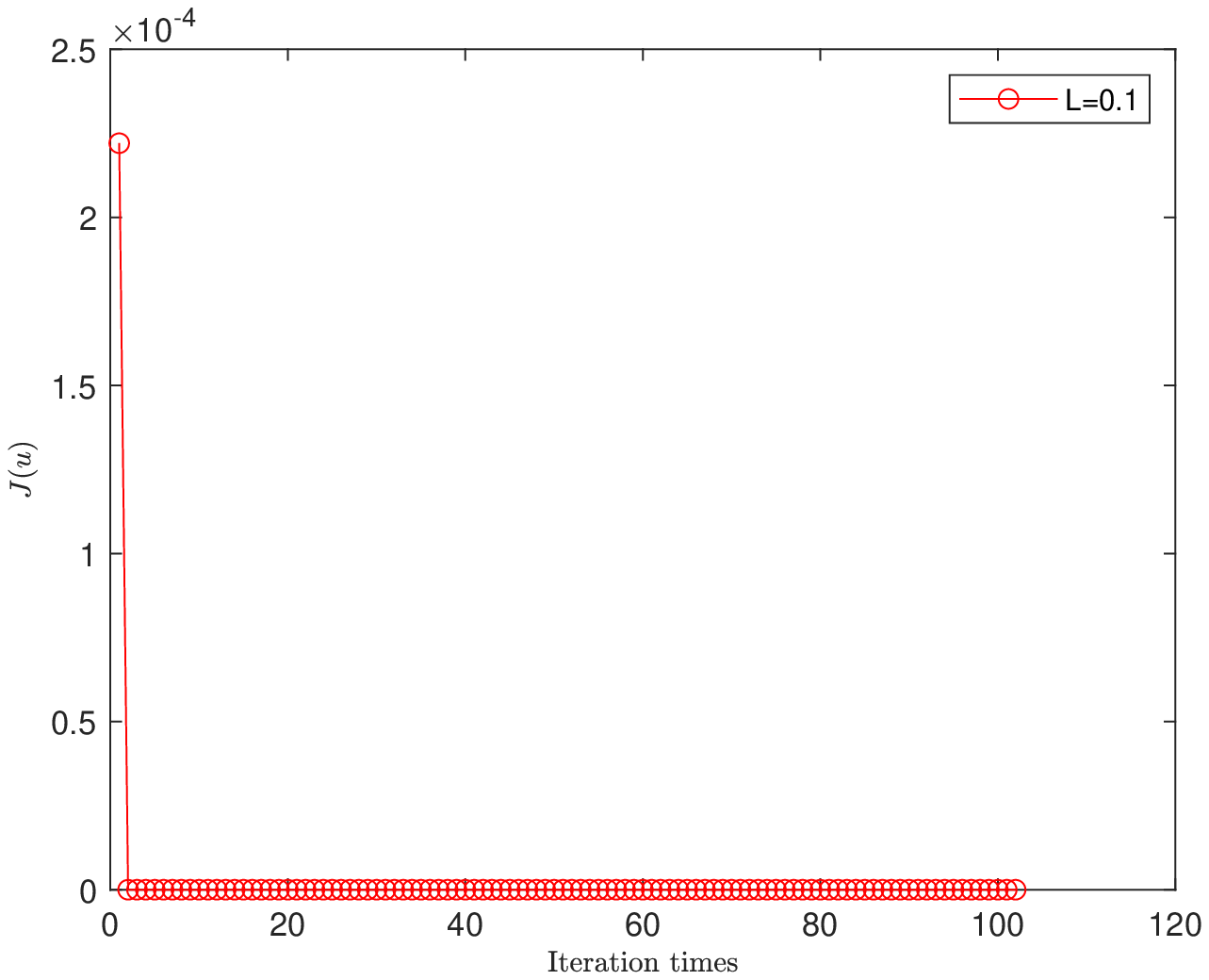}}
  \centerline{The case $L=0.1$.}
  \end{minipage}
\caption{The performance of Algorithm \ref{algorithm-MSA} corresponding to different choices of $L$.}
\label{figure-1}
\end{figure}

In Figure \ref{figure-1}, the graph on the left-hand side indicates that $J(u^{m-1}(\cdot))$ fluctuates up and down around the optimal cost $0$
as $m$ increases. The graph on the right-hand side indicates that $J(u^{0}(\cdot))=2.25 \times 10^{-4}$ at the initial time and it descends to the optimal cost $0$ rapidly after one iterative step, and then $J(u^{m-1}(\cdot))$ remains steady at $0$ as $m$ increases. One can note that the convergence is rapid and sharp.

In conclusion, when $L$ is relatively small, Algorithm \ref{algorithm-MSA} converges to the minimum of Example \ref{example-1} ($L=0.1$).
This demonstrates that the modified MSA can indeed help us find the optimum for some simple or specific stochastic recursive control problems.
However, it shows some fluctuation of the sequence $\{J(u^{m}(\cdot))\}_{m}$ around the optimal value when we set a relatively larger $L$ ($L=1$). The reason for this
phenomenon may involve the error of discretizing the time interval and the computational error of solving the FBSDE in Example \ref{example-1} by the numerical method.

\end{document}